\documentclass{article}
\usepackage{parskip}
\usepackage[margin=1.0in]{geometry}

\usepackage{graphicx}
\usepackage{caption}
\usepackage{subcaption}
\usepackage{tikz}
\usepackage{tikz-cd}
\usepackage{physics}
\usepackage{amsmath}
\usepackage{amsthm}
\usepackage{amstext}
\usepackage{amssymb}
\usepackage{thmtools}
\usepackage{bm}
\newtheorem{theorem}{Theorem}[section]
\newtheorem{corollary}[theorem]{Corollary}
\newtheorem{lemma}[theorem]{Lemma}
\newtheorem{prop}[theorem]{Proposition}

\theoremstyle{definition}
\newtheorem{defn}{Definition}[section]
\newtheorem{exmp}[theorem]{Example}

\usepackage[english]{babel}
\usepackage{pdfpages}

\usepackage{dsfont}
\usepackage{xcolor}
\definecolor{amber}{rgb}{1.0, 0.6, 0.0}
\usepackage{textcomp}
\usepackage{listings}
\usepackage{relsize}
\usepackage{multicol}
\usepackage{hyperref}
\hypersetup{
	colorlinks=true,
	linkcolor=black,
	filecolor=black,      
	urlcolor=black,
	citecolor=blue,
}
\usepackage[capitalise]{cleveref}
\crefname{figure}{Fig.}{Fig.}
\crefname{table}{Table}{Tables}
\crefname{equation}{Eqn.}{Eqns.}
\crefname{algocf}{Algorithm}{Algorithms}
\crefname{exmp}{Example}{Ex.}

\crefname{lemma}{Lemma}{Lemmas}
\crefname{prop}{Proposition}{Propositions}
\crefname{corollary}{Corollary}{Cor.}
\crefname{defn}{Definition}{Definitions}

\usepackage{url}
\usepackage{pifont}

\usepackage{bbm}
\usepackage{mathtools}
\usepackage{hhline}
\usepackage{tabularx}

\newcommand{\N}{\mathbb{N}}
\newcommand{\R}{\mathbb{R}}

\newcommand{\mcM}{\mathcal{M}}
\newcommand{\mcN}{\mathcal{N}}

\newcommand{\mcL}{\mathcal{L}}

\newcommand{\Cont}{\mathrm{Cont}}
\newcommand{\SCont}{\mathrm{SCont}}
\newcommand{\Diff}{\mathrm{Diff}}

\newcommand{\Lie}{\mathrm{Lie}}

\newcommand{\mfc}{\mathfrak{cont}}
\newcommand{\mfg}{\mathfrak{g}}

\newcommand{\mfsc}{\mathfrak{scont}}
\newcommand{\mfpd}{\mathfrak{pdiff}}
\newcommand{\mfX}{\mathfrak{X}}

\usepackage{authblk}

\title{Local Universal Splitting Integrators for Contact Hamiltonian Systems}
\author[1,2]{George A. Kevrekidis}
\date{May 2026 \\ LA-UR-26-23804}

\affil[1]{T-5, Los Alamos National Laboratory, Los Alamos, NM 87545, USA}
\affil[2]{Department of Applied Mathematics and Statistics, Johns Hopkins University, Baltimore, MD 21218, USA}

\begin{document}

\maketitle

\begin{abstract}
    Contact Hamiltonian systems extend symplectic Hamiltonian mechanics to dissipative settings while retaining geometric structure. We develop a structure-preserving splitting framework for contact Hamiltonian systems on $J^1(\R^n)$ based on two tractable classes of exact-contact subflows: strict contactomorphisms and prolonged diffeomorphisms. Our main theoretical result is that the Lie algebra generated by the corresponding strict and prolonged Hamiltonians contains all polynomial-in-$p$ Hamiltonians and is therefore dense, in the $C^r$ topology on compact sets, in the Lie algebra of smooth contact Hamiltonians. This yields a local universality result and contact splitting integrators built from exact strict and prolonged subflows. We then show how these subflows can be realized numerically by lifting symplectic integrators on $T^*\R^n$ and ODE integrators on $\R^n\times\R$. Finally, we illustrate the framework on a sequence of low-dimensional examples.
\end{abstract}

\section{Introduction}

Contact geometry is an important subfield of differential geometry, often presented as an `odd-dimensional counterpart' to symplectic geometry. From a physical perspective, contact geometry is the natural geometric framework for describing dissipative systems, while still retaining a Hamiltonian-like structure and the coordinate-freedom enjoyed by symplectic systems. The theory thus finds many natural applications in modelling non-equilibrium physical systems, including thermodynamics, mechanics, and optics \cite{kholodenko2013applications_contact,arnold2013mechanics,mrugala1991contact_thermo,grmela2014contact_thermo}. Contact transformations were originally studied by Sophus Lie \cite{lie_contact} as they are, in some sense, the most general type of equivalence transformation for partial differential equations, and are prominently featured in the theory of integrable systems \cite{olver1995equivalence}. The formalism of contact geometry has also been used in the theory of optimal control and optimization, and it may also provide a natural geometric language for aspects of mean-field Langevin dynamics, a common model in molecular dynamics simulations. At the same time, more formal mathematical research is also of interest, most notably in low-dimensional contact topology and knot theory \cite{geiges2001brief,geiges2006contact,massot2014topological}, but we will refrain from discussing these applications here.

Motivated by this wide range of applications, there has been renewed interest in developing the mathematical framework and accompanying structure-preserving numerical methods for contact systems, which faithfully capture their dissipative character \cite{de2019contact,simoes2020contact,de2023time,gaset2020contact}; perhaps surprisingly, the development of such methods has lagged significantly behind the development of symplectic integrators for Hamiltonian systems, which are a staple of structure-preserving numerical integration. In this work, we develop a splitting-based framework for contact Hamiltonian systems built from two tractable classes of exact contact subflows: strict contactomorphisms and prolonged diffeomorphisms. The theoretical ingredient is a Lie-algebra density result on jet spaces, while the numerical ingredient is that both classes admit concrete lifted realizations from symplectic and base-space ODE integrators.

\subsection{Literature Review}

Structure-preserving integrators for contact Hamiltonian systems were originally proposed in \cite{feng2010contact,feng2010symplectic}, where the correspondence between contact Hamiltonian systems on $\R^{2n+1}$ and homogeneous symplectic Hamiltonian systems on $\R^{2n+2}$ is exploited to construct (implicit) contact integrators (for a mathematical formulation of the latter, see also \cite[Appendix 4]{arnold2013mechanics}). The homogeneous structure of such `lifted' conservative systems is also directly exploited in \cite{araujo2025jacobi}, which studies the geometric integration of the more general Jacobi systems within a Poisson-manifold framework; contact systems are a special case. For contact systems specifically, a rigorous mathematical framework for the construction of contact variational integrators (of arbitrary order) is developed in \cite{bravetti2020numerical,vermeeren2019contact}, based on a discrete version of the Herglotz variational principle described in \cite{guenther1996herglotz}. At the same time, for separable contact Hamiltonian systems, splitting integrators have often been used \cite{zadra2023topics,zadra2021vdp} for numerical integration. However, as noted in \cite{mclachlan2002splitting}, due to the complexity of the corresponding contact ODEs, there are few contact Hamiltonians that admit closed-form splittings.

Contact structure-preserving perspectives and algorithms have also begun to emerge in the machine learning literature, where they can be used to understand the geometric properties of optimization algorithms \cite{francca2021dissipative,bravetti2023bregman}, or construct structure-preserving neural networks for learning contact Hamiltonian systems from data \cite{testa2025geometric}. 

A foundational treatment of contact geometry and the contactomorphism group can be found in the classical references \cite{guenther1996herglotz,krieg_michorl1997convenient,banyaga2013structure,geiges2006contact}, and in \cite{olver1995equivalence}, which stresses the connection to PDE equivalence. The notes by Ko Honda \cite{honda2019notes} also provide a comprehensive overview of the subject.

\subsection{Motivation}
The underlying observation motivating this work is that, while few contact Hamiltonian systems admit a closed-form splitting, there are two classes of contact subflows for which we already possess mature theoretical and numerical tools. These are (1) strict contactomorphisms and (2) prolonged diffeomorphisms, which we define precisely in the next section. Each class can be integrated while preserving the contact structure exactly, and each comes with an established lower-dimensional numerical toolkit: symplectic integrators for the strict part and standard ODE integrators, together with prolongation, for the prolonged part. In practice, many Hamiltonians of interest already decompose into strict and prolonged pieces, or require only a small number of commutator corrections, so these two tractable families already provide a broad practical design space.

Our main theoretical result is that the Lie algebra generated by the corresponding strict and prolonged Hamiltonians is dense, in the $C^r$ topology on compact sets, in the Lie algebra of smooth contact Hamiltonians on $J^1(\R^n)$. This yields a local universal splitting framework in a complementary sense: For already-structured Hamiltonians, one can work directly with strict and prolonged generators. For generic Hamiltonians, the polynomial-in-$p$ step serves mainly as a universality guarantee: one replaces the target Hamiltonian by a polynomial-in-$p$ surrogate, represents that surrogate using strict and prolonged generators, and then applies standard BCH-based splitting and commutator constructions.

\subsection{Contributions and Structure}

\paragraph{Overview.} We begin with a brief review of the minimal relevant mathematical background in \cref{sec:Background}, including the definitions of strict contactomorphisms and prolonged diffeomorphisms on the jet bundle. We then prove in \cref{sec:Lie_Density} that the Lie algebra generated by the corresponding strict and prolonged Hamiltonians is dense, in the $C^r$ topology on compact sets, in the Lie algebra of smooth contact Hamiltonians. In \cref{sec:Integrators} we combine this representation with standard BCH-based product-formula constructions to obtain a local universal splitting framework and asymptotic error estimates. In \cref{sec:Numerics} we show how the relevant exact subflows can be realized in practice by lifting symplectic integrators on $T^*\R^n$ and ODE integrators on $\R^n\times\R$, and we demonstrate the resulting methods on a range of low-dimensional examples. In \cref{sec:Discussion}, we summarize our findings and discuss their implications for geometric integration and for the study of contact systems more broadly.

\paragraph{Contributions.} The main contributions of this work are as follows:
\begin{itemize}
    \item We prove a Lie-algebra density result on jet spaces, showing that the Lie algebra generated by strict contact Hamiltonians and prolonged Hamiltonians is dense in the full Lie algebra of smooth contact Hamiltonians on $J^1(\R^n)$, in the $C^r$ topology on compact sets.
    \item We identify strict contactomorphisms and prolonged diffeomorphisms as a practically useful pair of exact-contact building blocks: many Hamiltonians already decompose into these pieces, or require only a small number of commutator corrections, before any approximation step is invoked.
    \item We use the same representation to construct a local universal splitting framework for the generic fallback case, in which polynomial approximants of contact Hamiltonians admit high-order contact one-step maps built from finite compositions of exact strict contactomorphisms and prolonged diffeomorphisms.
    \item We derive approximation bounds and asymptotic error estimates on compact sets and finite time intervals for the resulting splitting constructions.
    \item We show how the two abstract building blocks can be realized in practice by lifting symplectic integrators on $T^*\R^n$ and ODE integrators on $\R^n\times\R$ to the jet space, and we demonstrate the resulting methods on a range of low-dimensional examples.
\end{itemize}

\section{Mathematical Background}
\label{sec:Background}

\subsection{Definitions}

A \textbf{contact manifold} is a tuple $(\mcM,\xi)$, where $\mcM$ is an odd dimensional smooth manifold (of dimension $2n+1$) and $\xi=\ker\alpha$ is the kernel of a one-form $\alpha$ on $\mcM$ satisfying the non-degeneracy condition $\alpha\wedge (d\alpha)^n \neq 0$. The one-form $\alpha$ is called a \textbf{contact form}, and the distribution $\xi$ is called a \textbf{contact structure}.

\begin{exmp} A simple example of a contact manifold is $\R^{2n+1}$ with coordinates $(x^i,u,p_i)$, where $x\in\R^n$ are the `position' coordinates, $p\in\R^n$ are the `momentum' coordinates, and $u\in\R$ is an additional fiber coordinate. The contact form can be taken to be $\alpha = du - p_i dx^i$, which defines a contact structure on $\R^{2n+1}$. The non-degeneracy condition yields the standard volume form $\alpha\wedge (d\alpha)^n = du \wedge dx^1 \wedge dp_1 \wedge ... \wedge dx^n \wedge dp_n$.
\end{exmp}

A \textbf{contactomorphism} is a diffeomorphism $\psi$ that preserves the contact structure, i.e. $\psi^*\xi = \xi$. Equivalently, $\psi$ is a contactomorphism if there exists a scalar function $\lambda:\mcM\to\R$ such that $\psi^*\alpha = e^{\lambda}\alpha$. The term $e^\lambda$ is a non-vanishing positive\footnote{Depending on the geometric setting, negative conformal factors may also be allowed, as long as they remain non-vanishing.} function called the \textbf{conformal factor} of $\psi$. The set of all contactomorphisms on a contact manifold $(\mcM,\xi)$ forms a group under composition, the \textbf{group of contactomorphisms}. We denote the group of contactomorphisms by $\Cont(\mcM,\xi)$.

Contact manifolds have a `local normal form' given by the Darboux theorem, stating:

\begin{theorem}[Darboux] Let $(\mcM,\xi=\ker\alpha)$ be a contact manifold. Then, for any point $z\in\mcM$, there exists a neighborhood $U$ of $z$ and a diffeomorphism $\phi:U\to V\subset \R^{2n+1}$, where $V$ is an open subset with coordinates $(x^i,u,p_i)$ such that $\phi(z)=0$ and $\phi^*(du - p_idx^i)=\alpha\vert_U$.
\end{theorem}

Similar to the symplectic case, the Darboux theorem implies that all contact manifolds are locally indistinguishable from $\R^{2n+1}$, i.e. there are no local invariants of contact structures. 

\begin{exmp}[Jets] A more general example of a contact manifold is the first jet bundle $J^1(\mcN)$, where $\mcN$ is an $n$-dimensional smooth base manifold. A point of $J^1(\mcN)$ records the first-order data of a local function on $\mcN$: two local functions determine the same $1$-jet at $x\in\mcN$ if they agree at $x$ and have the same differential there. In local coordinates, a jet is written $(x^i,u,p_i)$, where $x$ is the base point, $u$ is the function value, and $p_i$ are the first-derivative coordinates. Every smooth function $f:\mcN\to\R$ therefore defines a natural section
\begin{equation*}
j^1f:\mcN\to J^1(\mcN), \qquad x\mapsto (x,f(x),df_x),    
\end{equation*}
which in coordinates is $j^1f(x)=(x^i,f(x),\partial_i f(x))$. The bundle $J^1(\mcN)$ has a natural contact structure given by the contact form $\alpha = du - p_i dx^i$, for which $(j^1f)^*\alpha = 0$. We often make the useful identification $J^1(\mcN)\cong T^*\mcN \times \R$.
\end{exmp}

Following this construction, it is easy to see that $J^1(\R^n)\cong \R^{2n+1}\cong T^*\R^n\times\R$, equipped with the standard contact form $\alpha = du - p_i dx^i$. Thus, jets have global Darboux coordinates, and are the natural setting for the study of contact systems. In this work, we restrict all global approximation and splitting statements to $J^1(\R^n)$ with these global Darboux coordinates. For a general contact manifold, the same formulas should instead be read only chart-locally in a Darboux neighborhood, and therefore depend on the chosen chart; we do not address the additional patching issues that can arise when trying to glue such local constructions into a global splitting framework on a manifold with nontrivial topology. See \cref{ap:Coordinate_Free_Contact} for the corresponding coordinate-free definitions of the Reeb field, contact vector fields, and contact Hamiltonian vector fields.

In coordinates, we denote 
\begin{equation}
    z = (x,u,p) = (x^1,...,x^n,u,p_1...p_n) \in J^1(\R^n).
\end{equation}

Let $\mfX(J^1(\R^n))$ denote the space of vector fields on $J^1(\R^n)$. A vector field $X\in\mfX(J^1(\R^n))$ is called a \textbf{contact vector field} if it satisfies the Lie-derivative condition $\mathcal{L}_X\alpha = \zeta\alpha$ for some scalar function $\zeta$. The flow along a contact vector field is a one-parameter family of contactomorphisms.

Any contact flow $\Phi_X^t$ along a contact vector field $X\in\mfX(J^1(\R^n))$ is locally generated by a \textbf{contact Hamiltonian}, which is a smooth function $H:J^1(\R^n)\to\R$ via the formula
\begin{equation}
X_H = \frac{\partial H}{\partial p_i}\frac{\partial}{\partial x^i} - \qty(\frac{\partial H}{\partial x^i} + p_i\frac{\partial H}{\partial u})\frac{\partial}{\partial p_i} + \qty(p_i\frac{\partial H}{\partial p_i} - H)\pdv{u}.
\end{equation}

Under this convention, if the conformal factor of the flow $\Phi_H^t$ is defined by
\[
(\Phi_H^t)^*\alpha = e^{\lambda_t}\alpha,
\]
then
\begin{equation}
    \mathcal{L}_{X_H}\alpha = -\pdv{H}{u}\alpha\qc
    \partial_t \lambda_t(z) = -\pdv{H}{u}\qty(\Phi_H^t(z)).
\end{equation}

This correspondence identifies the tangent space at the identity of the contactomorphism group with the space of contact Hamiltonian vector fields, namely $T_{\mathrm{id}}\Cont(J^1(\R^n))\cong \mfc(J^1(\R^n))$, and therefore with the space of smooth functions on $J^1(\R^n)$.\footnote{There are different choices of coordinate-form for the contact vector field, which correspond to different choice of the reference one-form (where here we set $\alpha=du-pdx)$. This is similar to choosing which component receives a negative sign in the symplectic case. In both cases, there is a coordinate transformation (Legendre, or H\'enon respectively) that relates the different conventions.} For convenience, we will often identify $H\mapsto X_H$ and write $H$ instead of $X_H$ when there is no risk of confusion. Similarly, we identify the flow $\Phi^t_H$ with the exponential of the Hamiltonian, and write $\Phi^t_H = \Phi^t_{X_H}$.

\subsection{Subgroups and subalgebras}

We denote the group of contactomorphisms by $\Cont(J^1(\R^n))$ and its associated Lie algebra of contact vector fields by $\mfc(J^1(\R^n))$. We are interested in approximating flows, i.e. maps in the identity component of the contactomorphism group - we denote this latter subgroup by $\Cont_0(J^1(\R^n))$.

The \textbf{contact-Jacobi bracket} in terms of Hamiltonians takes the form
\begin{equation}
    \qty[H,K]=\qty{H,K}+\pdv{H}{u}\qty[K-p\pdv{K}{p}]-\pdv{K}{u}\qty[H-p\pdv{H}{p}],
\end{equation}
where $\{H,K\}$ is the Poisson bracket on $T^*\R^n$ in the convention fixed below. We will be interested in two subgroups of $\Cont(J^1(\R^n))$ and their associated Lie subalgebras, which we now define.

\subsubsection{Strict Contactomorphisms}

\textbf{Strict contactomorphisms} are the subgroup $\SCont(J^1(\R^n))\subset \Cont(J^1(\R^n))$ of contactomorphisms that preserve the contact form exactly, i.e. $\psi^*\alpha = \alpha$. The associated Lie algebra consists of \textbf{strict contact vector fields}, which are vector fields $X$ satisfying $\mathcal{L}_X\alpha = 0$, and are generated by contact Hamiltonians $H(x,p)$ that do not depend on the $u$ coordinate, i.e.

\begin{equation}
    \mfsc(J^1(\R^n)) = \qty{H\in C^\infty(J^1(\R^n)) : \pdv{H}{u} = 0}.
\end{equation}

Thus, strict contactomorphisms are generated by Hamiltonians that depend only on the base and dual fiber coordinates $(x,p)$. At the level of autonomous Hamiltonian flows, they project to symplectic Hamiltonian flows on $T^*\R^n$. We have the following lemma:

\begin{lemma}
    Let $K\in C^\infty(T^*\R^n)$, define $H(x,u,p)=K(x,p)$, and let $\Phi_H^t$ and $\Phi_K^t$ denote the corresponding contact and symplectic Hamiltonian flows, respectively. Then $\Phi_H^t$ is a strict contactomorphism for each $t$ for which the flow exists, and the following diagram commutes:
\begin{equation}
        \begin{tikzcd}
        J^1(\R^n) \arrow[r, "\Phi_H^t"] \arrow[d, "\pi_{(x,p)}"'] & J^1(\R^n) \arrow[d, "\pi_{(x,p)}"] \\
        T^*\R^n \arrow[r, "\Phi_K^t"] & T^*\R^n
        \end{tikzcd}
\end{equation}
    where $\pi_{(x,p)}:J^1(\R^n)\to T^*\R^n$ is the natural projection map. Conversely, given a symplectic Hamiltonian flow $\Phi_K^t$ on $T^*\R^n$ generated by $K$, the contact Hamiltonian $H(x,u,p)=K(x,p)$ generates a strict contact flow $\Phi_H^t$ satisfying the same commutative diagram.
\end{lemma}

Strict contactomorphisms preserve the induced volume form $\alpha\wedge (d\alpha)^n$, and therefore are a natural subgroup of contactomorphisms to consider when constructing structure-preserving integrators. Moreover, since their Hamiltonian flows project to symplectic Hamiltonian flows on $T^*\R^n$, we can leverage the rich theory of symplectic integrators \cite{hairer2006geometric} to construct high-order integrators for this subgroup. However, since strict contactomorphisms are a proper subgroup of the full contactomorphism group, we cannot expect to approximate general contact flows by compositions of strict contactomorphisms alone.

\subsubsection{Prolonged Diffeomorphisms}

\textbf{Prolonged diffeomorphisms} are the subgroup $\Diff^{(1)}(J^1(\R^n))\subset \Cont(J^1(\R^n))$ of contactomorphisms that arise as the prolongation of a diffeomorphism on the $(x,u)$ variables. The associated Lie algebra of prolonged vector fields is generated by Hamiltonians that are affine-in-$p$, that is:

\begin{align}
    \mfpd(J^1(\R^n)) &= \qty{H\in C^\infty(J^1(\R^n)) : H(x,u,p) = f(x,u)+g^i(x,u)p_i}
\end{align}
where $f,g^i$ are smooth functions on the base and fiber $(x,u)$ coordinates. Indeed, one can verify that, under such a Hamiltonian, the evolution of the $(x,u)$-component of the contact flow `decouples' from that of $p$. More formally, at the level of autonomous flows:

\begin{lemma}
    Let $H(x,u,p)=f(x,u)+g^i(x,u)p_i$, and let $\Phi_H^t$ denote the corresponding contact flow on $J^1(\R^n)$. Then there exists a flow $\varphi^t:\R^n\times\R\to\R^n\times\R$ generated by the vector field
    \[
        Y = g^i(x,u)\pdv{x^i} - f(x,u)\pdv{u}
    \]
    such that the following diagram commutes:
\begin{equation}
        \begin{tikzcd}
        J^1(\R^n) \arrow[r, "\Phi_H^t"] \arrow[d, "\pi_{(x,u)}"'] & J^1(\R^n) \arrow[d, "\pi_{(x,u)}"] \\
        \R^n\times\R \arrow[r, "\varphi^t"] & \R^n\times\R
        \end{tikzcd}
\end{equation}
    where $\pi_{(x,u)}:J^1(\R^n)\to \R^n\times\R$ is the natural projection map. In particular, each fixed-time map $\Phi_H^t$ is a prolonged diffeomorphism. Conversely, given a flow $\varphi^t$ on $\R^n\times\R$ generated by
    \[
        Y = g^i(x,u)\pdv{x^i} - f(x,u)\pdv{u},
    \]
    its first prolongation to $J^1(\R^n)$ is the contact flow generated by the Hamiltonian $H(x,u,p)=f(x,u)+g^i(x,u)p_i$.
\end{lemma}

Prolonged diffeomorphisms may not preserve the contact form. Here, the momentum variables $p$ act as `derivatives' of $u$ with respect to the base variables $x$, and can be determined by the chain rule (i.e. by prolongation) from the transformation of the base variables. Thus, one may use an arbitrary integrator for the base variables, and then determine the transformation of the momentum variables by prolongation. This is a powerful tool for constructing contact integrators, as it allows us to leverage the rich theory of numerical integration for ODEs on $\R^n\times\R$ to construct integrators for this subgroup. However, since prolonged diffeomorphisms are also a proper subgroup of the full contactomorphism group, we cannot expect to approximate general contact flows by compositions of prolonged diffeomorphisms alone.

\subsubsection{Polynomial-in-\texorpdfstring{$p$}{p} Subalgebras}

Let us denote by $P^{(k)}(J^1(\R^n))$ the vector space of Hamiltonians that are polynomial in $p$ of degree at most $k$, i.e. 
\begin{equation}
    P^{(k)}(J^1(\R^n)) = \qty{H\in C^\infty(J^1(\R^n)) : H(x,u,p) = \sum_{|\alpha|\leq k} f_\alpha(x,u)p^\alpha},
\end{equation}
where $\alpha$ is a multi-index and $f_\alpha$ are smooth functions on the $(x,u)$-base. Note that $P^{(1)}(J^1(\R^n))$ coincides with the Lie algebra of affine-in-$p$ Hamiltonians, and therefore generates the subgroup of prolonged diffeomorphisms. Otherwise, the $P^{(k)}(J^1(\R^n))$ are not Lie subalgebras, but form a filtration, by \cref{lem:bracket_degree}:
\begin{equation}
    \qty[P^{(k)}(J^1(\R^n)),P^{(m)}(J^1(\R^n))] \subset P^{(k+m)}(J^1(\R^n)).
\end{equation}
Thus, the union $P$ of all $P^{(k)}(J^1(\R^n))$:
\begin{align}
    P(J^1(\R^n)) & = \bigcup_{k=0}^\infty P^{(k)}(J^1(\R^n)),
\end{align}
forms a Lie subalgebra of $\mfc(J^1(\R^n))$. The main utility of this large subalgebra is the following $C^r$ polynomial density result:

\begin{prop}[Polynomial Density]
    \label{thm:density}
    Any smooth function $f \in C^{\infty}(J^1(\R^n))$ can be approximated with respect to the $C^r$ norm, for any $r<\infty$, on compact sets by functions in $P(J^1(\R^n))$. That is, for a compact set $U\subset J^1(\R^n)$ and any $\epsilon > 0$, there exists a function $\hat{f}_\epsilon\in P(J^1(\R^n))$ such that:
    \begin{equation}
        \norm{f - \hat{f}_\epsilon}_{C^r(U)} < \epsilon.
    \end{equation} Thus, the closure satisfies 
    \begin{equation}
        C^\infty(U) \subset \overline{P(U)}^{C^r(U)},
    \end{equation}
    i.e. $P(U)$ is dense in $C^\infty(U)$ with respect to the topology induced by the $C^r(U)$ norm. Here, $P(U) = \qty{f|_U : f\in P(J^1(\R^n))}$ is the restriction of $P(J^1(\R^n))$ to $U$.

    \begin{proof}
        This is a classical result that is a $C^r$ version of the Stone-Weierstrass theorem. See \cite[Theorem 1.6.2]{narasimhan1985analysis} for the specific version.
    \end{proof}
\end{prop}

\section{Lie Density}
\label{sec:Lie_Density}

As we mention in the previous section, polynomial-in-$p$ Hamiltonians can already approximate any smooth contact Hamiltonian over a compact set. This polynomial step is primarily a formal universality device: it shows that the Lie algebra generated by strict contact Hamiltonians and prolonged Hamiltonians is large enough to recover arbitrary smooth Hamiltonians in the $C^r$ topology on compacta. From a practical numerical viewpoint, however, the more important observation is that many Hamiltonians of interest already decompose into strict and prolonged pieces, or into sums of such pieces together with a small number of commutator-generated corrections. In particular, nonpolynomial but separable terms, such as relativistic kinetic energies like $\sqrt{1+p^2}$, still fit naturally into the framework; the main obstruction is genuinely mixed $(p,u)$ dependence.

Our strategy is therefore twofold. First, we show that the Lie algebra generated by strict and prolonged Hamiltonians contains all polynomial-in-$p$ Hamiltonians, and therefore is dense in the full Lie algebra of smooth contact Hamiltonians on $J^1(\R^n)$. Second, we leverage the same generators directly in the already-separable cases that occur most naturally in applications.

Indeed, let
\begin{equation}
    \mfg \doteq \Lie\qty(\mfsc(J^1(\R^n)),\mfpd(J^1(\R^n))),
\end{equation}
be the smallest Lie subalgebra of $\mfc(J^1(\R^n))$ containing the strict and prolonged contact Hamiltonians. That is, we define $\mfg$ inductively as follows:
\begin{align*}
    \mfg_0 & = \mfsc(J^1(\R^n))+\mfpd(J^1(\R^n)) \\
    \mfg_{k+1} & = \mfg_k + [\mfsc(J^1(\R^n)),\mfg_k] + [\mfpd(J^1(\R^n)),\mfg_k]\\
    \mfg & = \bigcup_{k=0}^\infty \mfg_k
\end{align*}
where, for two vector subspaces $\mathfrak{a},\mathfrak{b}\subset \mfc(J^1(\R^n))$, we define their Lie bracket as the vector space generated by the span of the Lie brackets of their elements, i.e.
\begin{align*}
    [\mathfrak{a},\mathfrak{b}] = \mathrm{span}\qty{\qty[a,b] : a\in\mathfrak{a}, b\in\mathfrak{b}}
\end{align*}

\begin{prop}
    \label{prop:subalgebra_density} 
    The Lie algebra $\mfg$ generated by strict contact Hamiltonians and prolonged Hamiltonians contains all polynomial-in-$p$ Hamiltonians, i.e. $P^{(k)}(J^1(\R^n))\subset \mfg$ for all $k\in\N$. Therefore, the closure $\overline{\mfg}$ is dense in $\mfc(J^1(\R^n))$ in the $C^r$ topology on compact sets.
    \begin{proof}(Sketch) Any constant-coefficient monomial $\gamma p^\alpha$ is contained in the strict subalgebra. It is then sufficient to show that the operations of degree-raising, degree lowering, and multiplication by scalar functions $f(x,u)$ can be generated by Lie brackets with strict and prolonged Hamiltonians (\cref{lem:bracket_operators}). In particular, the scalar-multiplication step produces the desired leading term $f(x,u)p^\alpha$ together with lower-degree terms. These lower-degree terms can then be removed recursively, beginning with the highest remaining degree, which yields a triangular elimination procedure. Iterating this argument shows that any polynomial-in-$p$ Hamiltonian can be generated by Lie brackets of strict and prolonged Hamiltonians, and in fact by brackets of depth at most one.
    \end{proof}
\end{prop}

Our proof is slightly stronger, implying that the only element of $\mfsc$ needed is in fact $p^\alpha$ for $\abs{\alpha}=2$, which would give a graded algebraic structure, albeit at the cost of deeper brackets. On the other hand, with our current choice of generators, we can obtain any polynomial-in-$p$ Hamiltonian with brackets of depth at most one, which can be a significant advantage for numerical applications. This large generator set is useful precisely because many Hamiltonians already encountered in practice decompose naturally into strict and prolonged pieces, so no preliminary polynomial approximation is needed in those cases. The polynomial-density argument should therefore be read primarily as a universality statement, while the practical advantage of the framework comes from the breadth of already integrable pieces it admits.

\textit{Remark:} This result extends to time-dependent Hamiltonians, where $H(x,u,p,t)$ is a smooth function of time as well. In this case, the Lie algebra generated by strict and prolonged Hamiltonians contains all polynomial-in-$p$ Hamiltonians that are also smooth in time, and therefore is dense in the space of all smooth time-dependent Hamiltonians. We only touch on a time-dependent example in \cref{sec:Numerics}, however, the generalization of our results follows the standard time-dependent geometric integration theory, where the time variable is treated as an additional coordinate, and the Hamiltonian is lifted to a time-independent Hamiltonian on an extended phase space \cite{arnold2013mechanics,hairer2006geometric,guenther1996herglotz}.

\section{Local Universal Splittings}
\label{sec:Integrators}

\subsection{Qualitative Construction}

We give a qualitative description of the construction of our `universal splitting integrators', followed by a more precise discussion of error control and convergence properties.

We assume that $H\in C^{\infty}(J^1(\R^n))$ is a smooth contact Hamiltonian whose flow $\Phi^t_H$ we are interested in approximating over a compact subset $U\subset J^1(\R^n)$. There are two conceptually distinct uses of the framework. For many Hamiltonians arising in practice, one can work directly with a decomposition into strict and prolonged pieces, possibly together with a small number of depth-one commutator identities, in which case no polynomial surrogate is required. The polynomial approximation step enters primarily as a formal universality argument for generic smooth $H$ on compact sets: it guarantees that, even when no such decomposition is available a priori, one may replace $H$ by a polynomial-in-$p$ surrogate $H^{(N)}$ and then apply the Lie-density result.

The first step in this universal-construction route consists of obtaining a suitable $H^{(N)}\in P^{(N)}(U)$ for some $N<\infty$ satisfying:
\begin{align*}
    \norm{H - H^{(N)}}_{C^r(U)} < \varepsilon.
\end{align*}
This is a standard polynomial approximation problem. The existence of such a function is guaranteed by \cref{thm:density}, and a constructive approach can be established by using e.g. multivariate Bernstein or Chebyshev polynomials \cite{devore1993constructive}.\footnote{The specific choice of polynomial approximation, together with the geometry of $U$ and regularity of $H$, determines both the decay of the approximation error $\varepsilon_N$ and the derivative bounds of the surrogate $H^{(N)}$, and therefore influences the computational cost as well as the constants appearing in the final global error estimate. In practice, one may want to choose a polynomial approximation that minimizes the degree needed to achieve a given error tolerance $\varepsilon$. For the two particular choices of approximation basis, a straightforward generalization of the univariate approximation results apply to compact boxes in $\R^{2n+1}$, where we can define a box $Q = \prod_{i=1}^{2n+1} [a_i,b_i]$ such that $U\subset Q$. We can then use the multivariate Bernstein or Chebyshev polynomials defined on $Q$ and restricted to $U$.}
From a practical standpoint, this polynomial stage is the main limitation of the generic pipeline, since large degrees may be required once many variables are involved. By contrast, nonpolynomial but separable Hamiltonians, including relativistic-like kinetic terms such as $\sqrt{1+p^2}$, already fit naturally into the strict/prolonged splitting philosophy; the genuinely difficult terms are those with nontrivial mixed $(p,u)$ dependence.

Since $H^{(N)}$ is a polynomial-in-$p$ Hamiltonian, we can express it as a finite sum of monomials in $p$ with smooth coefficients in $(x,u)$. By the subalgebra density result of \cref{prop:subalgebra_density} (and its associated proof), we can express $H^{(N)}$ as a finite sum of Lie brackets of strict and prolonged Hamiltonians of depth at most one; there exists a finite set of generators $\qty{s_i\in\mfsc(J^1(\R^n))}$ and $\qty{d_i\in\mfpd(J^1(\R^n))}$, $i\in I$, such that:
\begin{equation}
    H^{(N)} = s_0 + d_0 + \sum_{i\in I} [s_i,d_i]
\end{equation}
where $I$ is a finite index set. It then remains to relate the flow $\Phi^t_{H^{(N)}}$ to the flows of the generators $s_i$ and $d_i$.  

If a Hamiltonian $H$ is expressed as a sum of $K$ Hamiltonians $H = \sum_{i=1}^K H_i$, then, standard results in geometric integration theory relying on the Baker-Campbell-Hausdorff formula (\cref{lem:bch,cor:contact_BCH}) allow us to construct a splitting integrator for $H$ in terms of the flows of the $H_i$. For example:

\begin{exmp}[Standard Integrators] Suppose that $H=A+B$. The standard Lie Trotter ($k=1$) and Strang ($k=2$) splitting integrators take the form:
\begin{align}
    \text{Lie-Trotter: }\Phi^t_H & = \exp(tA)\exp(tB) + O(t^2) \\
    \text{Strang: }\Phi^t_H & = \exp(\frac{t}{2}A)\exp(tB)\exp(\frac{t}{2}A) + O(t^3)
\end{align}
Note that while the exact error constants will depend on the specific bracket that governs the (here contact) Lie algebra, the construction is valid for any Lie algebra, and these are essentially identical to the splittings for symplectic systems.
\end{exmp}

Any two-term splitting can be applied recursively for Hamiltonians with multiple terms, and higher order integrators can be constructed by the classical Yoshida and Suzuki constructions \cite{suzuki1990fractal,yoshida1990construction}, which are based on the same principle of canceling out leading order error terms by choosing the coefficients of the symmetric compositions of second-order integrators. 

Up to this point, we therefore obtain a bona fide `outer' splitting integrator for $H^{(N)}$ in terms of the flows of $s_0,d_0$ and the flows generated by the brackets $[s_i,d_i]$.
In order to obtain a splitting in terms of the flows of our \textit{specific generators}, we can use the Baker-Campbell-Hausdorff formula once again to express the flow of $[s_i,d_i]$ as a composition of the flows of $s_i$ and $d_i$, up to an error term that depends on higher-order brackets. A simple $O(t^{3/2})$ splitting is given by:

\begin{exmp}[Commutator Splitting] Let $s_i\in\mfsc$ and $d_i\in\mfpd$. Then, we have the following splitting:
\begin{equation}
    \Phi^t_{[s_i,d_i]} = \exp(\sqrt{t} s_i)\exp(\sqrt{t} d_i)\exp(-\sqrt{t} s_i)\exp(-\sqrt{t} d_i) + O(t^{3/2}).
\end{equation}
\end{exmp}

Two standard product-formula constructions underlie the argument. First, when a Hamiltonian is written as a sum of exactly integrable pieces, the Baker-Campbell-Hausdorff formula (BCH, \cref{cor:contact_BCH}) yields the familiar Lie--Trotter, Strang, and higher-order symmetric splittings. Second, when a term is written as a commutator, its flow can be approximated by standard commutator gadgets built from the exact subflows of the generators. Since the representation of $H^{(N)}$ above involves only sums of generators and depth-one commutator terms, recursive use of these two templates yields arbitrary-order contact one-step maps:

\begin{prop}[High-Order Splittings from the Depth-One Representation]
    \label{prop:Lie_Splitting_Integrators}
    Let
    \begin{equation}
        H^{(N)} = s_0 + d_0 + \sum_{i\in I}[s_i,d_i],
        \qquad s_i\in\mfsc(J^1(\R^n)), \quad d_i\in\mfpd(J^1(\R^n)),
    \end{equation}
    be given in the depth-one representation above, and let $U\subset J^1(\R^n)$ be compact.
    Assume that, for all sufficiently small positive and negative times required by the composition coefficients, the exact subflows of each generator $s_i$ and $d_i$ are well-defined on a neighborhood of $U$, and that all finite compositions used below remain in a fixed compact subset of that neighborhood.

    Then, for any order $k\in\N$, there exists a contact one-step map $\Psi^{(k)}_{N,h}$
    given by a finite composition of exact flows of the strict and prolonged generators such that
    \begin{equation}
        \Psi^{(k)}_{N,h} = \Phi_{H^{(N)}}^h + O(h^{k+1})
    \end{equation}
    in the $C^r(U)$ topology for every finite $r$, as $h\to 0$.
\end{prop}
\begin{proof}[Sketch]
    The construction proceeds in two layers.

    First, view
    \begin{equation*}
        H^{(N)} = s_0 + d_0 + \sum_{i\in I}[s_i,d_i]
    \end{equation*}
    as a sum of exactly integrable pieces, namely the generator flows of $s_0$, $d_0$, and the commutator flows $\Phi^h_{[s_i,d_i]}$. For such sums, the BCH formula in the contact setting, as recorded in \cref{cor:contact_BCH}, yields the usual Lie--Trotter, Strang, and higher-order symmetric product formulas. In particular, by the standard Yoshida--Suzuki recursive constructions \cite{suzuki1990fractal,yoshida1990construction}, one obtains an `outer' splitting of arbitrary order for the flow of $H^{(N)}$, provided the commutator flows $\Phi^h_{[s_i,d_i]}$ are themselves available.

    Second, for each commutator term $[s_i,d_i]$, one applies BCH once again at the level of the generators $s_i$ and $d_i$. Standard commutator gadgets built from the exact subflows of $s_i$ and $d_i$ approximate $\Phi^h_{[s_i,d_i]}$ to any prescribed order; see also the contact-specific low-order constructions in \cite[Propositions 3.2, 3.3]{zadra2021vdp}. Replacing every commutator flow in the outer splitting by such a gadget produces a finite composition involving only exact flows of strict and prolonged generators.

    By construction, the BCH expansion of this full composition agrees with $hH^{(N)}$ through order $k$, so the resulting one-step map satisfies
    \begin{equation*}
        \Psi^{(k)}_{N,h} = \Phi^h_{H^{(N)}} + O(h^{k+1})
    \end{equation*}
    on the fixed compact set $U$. Since every factor in the composition is a contactomorphism, the full composition is again a contactomorphism.
\end{proof}

\textit{Remark.} \cref{prop:Lie_Splitting_Integrators} is stated for compositions of exact strict and prolonged subflows. In practical realizations, these exact factors may be replaced by analytic exact-contact one-step maps whose local errors are chosen so that the aggregate realization error remains of order $O(h^{k+1})$ in the final composition. In that case, the resulting method is again an analytic exact-contact one-step map, and the additional realization error is absorbed into the same local truncation estimate. Consequently, under the analyticity assumptions invoked below, the same backward-error and modified conformal-factor conclusions apply to the realized method. We do not formalize this extension here for nonanalytic realizations or for lifts that are only approximately strict.

\subsection{Error Control}

We now proceed with a formal statement of the error estimates associated with our proposed universal splitting integrators. Let $U\subset J^1(\R^n)$ be a compact set, with a Hamiltonian $H\in C^{r+1}(U)$, and let
$H^{(N)}\in P^{(N)}(J^1(\R^n))$ be a polynomial-in-$p$ approximation satisfying:
\begin{equation}
    \varepsilon_N := \norm{H-H^{(N)}}_{C^{r+1}(U)}.
\end{equation}
There are two distinct ingredients in the argument. First, for the polynomial surrogate $H^{(N)}$, the depth-one representation together with the BCH-based product-formula constructions from \cref{lem:bch,cor:contact_BCH,prop:Lie_Splitting_Integrators} yields a high-order contact one-step map with controlled local error. Second, the exact flows of $H$ and $H^{(N)}$ are compared by the Gr\"onwall estimate of \cref{prop:gronwall}. The next theorem simply combines these two bounds in sequence.

\begin{theorem}[Local Universal Approximation]
Let $U\subset J^1(\R^n)$ be compact, let $H\in C^{r+1}(U)$, and let
\begin{equation}
    \varepsilon_N := \norm{H-H^{(N)}}_{C^{r+1}(U)},
    \qquad H^{(N)}\in P^{(N)}(J^1(\R^n)).
\end{equation}
Assume that the flows of $H$ and $H^{(N)}$ exist on $[0,T]$ and remain in a compact subset $U'\subset U$. By \cref{prop:subalgebra_density}, choose a representation
\begin{equation}
    H^{(N)} = s_0 + d_0 + \sum_{i\in I}[s_i,d_i],
    \qquad s_i\in\mfsc(J^1(\R^n)), \quad d_i\in\mfpd(J^1(\R^n)).
\end{equation}
Assume moreover that, for all sufficiently small positive and negative times required by the composition coefficients, the exact subflows of each generator $s_i$ and $d_i$ are well-defined on a neighborhood of $U'$, and that all finite compositions used in the construction remain in a fixed compact subset of that neighborhood.

Then, for any target order $p\in\N$, there exists a contact one-step map $\Psi^{(p)}_{N,h}$, given by a finite composition of exact flows of the strict and prolonged generators, and a constant $C_{\mathrm{split}}(N,p,T,r)>0$ such that
\begin{equation}
    \norm{\Phi_{H^{(N)}}^{nh} - \qty(\Psi^{(p)}_{N,h})^n}_{C^r(U)}
    \leq C_{\mathrm{split}}(N,p,T,r)\, h^p
\end{equation}
for all sufficiently small $h>0$ and all $nh\leq T$.

Consequently, there exists a constant $C_{\mathrm{approx}}(H,U,T,r)>0$ such that
\begin{equation}
    \norm{\Phi_H^{nh} - \qty(\Psi^{(p)}_{N,h})^n}_{C^r(U)}
    \leq C_{\mathrm{approx}}(H,U,T,r)\,\varepsilon_N
    + C_{\mathrm{split}}(N,p,T,r)\, h^p
\end{equation}
for all sufficiently small $h>0$ and all $nh\leq T$.

In particular, given any $\varepsilon>0$, one may first choose $N$ so that
\begin{equation}
    C_{\mathrm{approx}}\,\varepsilon_N < \varepsilon/2,
\end{equation}
and then choose $h>0$ so that
\begin{equation}
    C_{\mathrm{split}}\,h^p < \varepsilon/2.
\end{equation}
Hence, for any $\varepsilon>0$ and any target order $p$, there exist $N<\infty$ and $h>0$ such that
\begin{equation}
    \norm{\Phi_H^{nh} - \qty(\Psi^{(p)}_{N,h})^n}_{C^r(U)} < \varepsilon
\end{equation}
for all $nh\leq T$.
\end{theorem}

\textit{Remark.} The two constants in the theorem encode different sources of error. The approximation constant $C_{\mathrm{approx}}$ is the Gr\"onwall-type comparison constant from \cref{prop:gronwall} applied to the pair $(H,H^{(N)})$; in particular, it depends on the time horizon $T$, the regularity index $r$, the confinement neighborhood on which both flows remain defined, and uniform Lipschitz and derivative bounds for the corresponding contact vector fields. The splitting constant $C_{\mathrm{split}}$ is the standard global-error constant for the chosen order-$p$ one-step composition applied to the fixed depth-one representation of $H^{(N)}$; accordingly, it depends on $T$, $r$, the same confinement neighborhood, the target order $p$, the number and size of the generators in the chosen representation, and uniform bounds on the subflows and commutator data entering the BCH-based construction. Thus, the dependence on $N$ enters through the chosen polynomial surrogate $H^{(N)}$ and its representation, rather than through an independent universal constant.

\begin{proof}
The proof is a concatenation of three standard steps.

First, fix the chosen depth-one representation of $H^{(N)}$. By \cref{prop:Lie_Splitting_Integrators}, whose construction is based on the BCH/product-formula arguments recalled in \cref{lem:bch,cor:contact_BCH}, for any target order $p$ there exists a contact one-step map $\Psi^{(p)}_{N,h}$ with local error
\begin{equation}
    \Psi^{(p)}_{N,h} = \Phi_{H^{(N)}}^h + O(h^{p+1})
\end{equation}
in the $C^r(U')$ topology.

Second, standard local-to-global convergence theory for one-step methods on compact time intervals upgrades this local estimate to the global splitting bound
\begin{equation}
    \norm{\Phi_{H^{(N)}}^{nh} - \qty(\Psi^{(p)}_{N,h})^n}_{C^r(U)}
    \leq C_{\mathrm{split}}(N,p,T,r)\, h^p
\end{equation}
for all sufficiently small $h>0$ and all $nh\leq T$. This is the error coming from replacing the exact flow of the polynomial Hamiltonian $H^{(N)}$ by the splitting map, and the corresponding constant depends on the fixed depth-one representation, the target order, the confinement neighborhood, and the usual stability bounds entering one-step convergence theory.

Third, the error coming from replacing the original Hamiltonian $H$ by the polynomial surrogate $H^{(N)}$ is controlled by the Gr\"onwall-type comparison estimate of \cref{prop:gronwall}. Applied to the pair $H$ and $H^{(N)}$, it yields
\begin{equation}
    \norm{\Phi_H^{nh} - \Phi_{H^{(N)}}^{nh}}_{C^r(U)}
    \leq C_{\mathrm{approx}}(H,U,T,r)\,\varepsilon_N
\end{equation}
for all sufficiently small $h>0$ and all $nh\leq T$.
Here one simply applies \cref{prop:gronwall} to the pair $(H,H^{(N)})$, with comparison size $\varepsilon_N = \norm{H-H^{(N)}}_{C^{r+1}(U)}$.

Combining the splitting estimate and the approximation estimate by the triangle inequality,
\begin{equation*}
    \norm{\Phi_H^{nh} - \qty(\Psi^{(p)}_{N,h})^n}_{C^r(U)}
    \leq
    \norm{\Phi_H^{nh} - \Phi_{H^{(N)}}^{nh}}_{C^r(U)}
    +
    \norm{\Phi_{H^{(N)}}^{nh} - \qty(\Psi^{(p)}_{N,h})^n}_{C^r(U)},
\end{equation*}
we obtain
\begin{equation*}
    \norm{\Phi_H^{nh} - \qty(\Psi^{(p)}_{N,h})^n}_{C^r(U)}
    \leq C_{\mathrm{approx}}(H,U,T,r)\,\varepsilon_N
    + C_{\mathrm{split}}(N,p,T,r)\, h^p,
\end{equation*}
which is the claimed bound.
\end{proof}

The next corollary is stated in the main text because, in the contact setting, conformal-factor tracking plays the role that modified-energy control plays in the symplectic setting. It is a direct consequence of the standard analytic BEA and conformal-tracking estimates recorded in the appendix, and is included here to make that contact-specific implication explicit.

\begin{corollary}[Modified Conformal Factor]
Assume in addition that $H^{(N)}$ is real analytic and that the one-step map
$\Psi^{(p)}_{N,h}$ above is analytic of local order $p$ for the Hamiltonian
$H^{(N)}$. Then, by \cref{thm:analytic_BEA,thm:conformal_tracking}, there exists a
modified Hamiltonian
\begin{equation}
    \tilde{H}_{N,h} = H^{(N)} + O(h^p)
\end{equation}
whose flow shadows $\Psi^{(p)}_{N,h}$ on fixed finite time intervals, and whose conformal
factor $\tilde{\lambda}_{N,h}(t)$ satisfies
\begin{equation}
    \tilde{\lambda}_{N,h}(t) = \lambda_N(t) + O(h^p t),
    \qquad 0\leq t\leq T,
\end{equation}
where $\lambda_N(t)$ is the conformal factor of the exact flow of $H^{(N)}$.

If, moreover, the polynomial approximation is chosen so that
$\varepsilon_N = O(h^p)$, then the same estimate compares directly to the conformal
factor of the original Hamiltonian $H$:
\begin{equation}
    \tilde{\lambda}_{N,h}(t) = \lambda_H(t) + O(h^p t).
\end{equation}
\end{corollary}

\section{Numerics}
\label{sec:Numerics}

\subsection{Lifting Integrators}

The previous section was abstract: after approximating a target Hamiltonian by a polynomial-in-$p$ surrogate and expressing that surrogate in depth-one form, we obtained high-order contact one-step maps from finite compositions of exact strict and prolonged subflows, together with the corresponding modified conformal-factor conclusion in the analytic setting. We now turn to a separate practical question: how can these two classes of contact maps be realized numerically? In this section, we explain in practice how strict contactomorphisms can be obtained by lifting symplectic integrators on $T^*\R^n$, and how prolonged diffeomorphisms can be obtained by lifting ODE integrators on $\R^{n+1}$ to the jet space $J^1(\R^n)$.

These constructions should be viewed as an implementation layer for the abstract framework above, rather than as a second approximation theorem. When exact generating functions or exact prolongations are available, the realized maps fall directly within the preceding framework; more generally, analytic exact-contact realizations of sufficiently high local order are absorbed into the same local truncation estimate, and hence into the same backward-error and modified conformal-factor picture. By contrast, realizations that are only approximately strict, or otherwise fall outside the analytic exact-contact setting, may still be useful in practice but are not covered by the formal theorem as stated. We then present a range of numerical examples, illustrating these realizations on low-dimensional contact Hamiltonian systems and comparing them to standard non-structure-preserving integrators.

\subsubsection{Lifting Symplectic Integrators}

Let $\varphi:(x,p)\mapsto(\bar{x},\bar{p})$ be a symplectomorphism on $(T^*\R^n, dx\wedge dp)$, and suppose that $G(x,p)$ is its associated generating function, satisfying: 
\begin{equation}
    \varphi^*(\bar{p} d\bar{x}) - pdx = dG.
\end{equation}
Then, the transformation $\Phi:(x,u,p)\mapsto(\bar{x},\bar{u},\bar{p})$ defined by
\begin{equation}
    \Phi(x,u,p) = (\bar{x}, u + G(x,p), \bar{p})
\end{equation}
is a strict contactomorphism on $J^1(\R^n)$. Thus, one may `lift' a symplectic integrator on $T^*\R^n$ to a strict contactomorphism on $J^1(\R^n)$ by first computing the flow on $(x,p)$ and then using the generating function of the symplectic integrator to define the transformation of the $u$ variable.

The generating function $G$ is unique, up to an additive constant, and can be computed explicitly for many standard symplectic integrators, such as the symplectic Euler and Störmer-Verlet methods, which allows us to construct explicit strict contactomorphisms corresponding to these integrators. More generally, if one defines a symplectic transformation via an implicit generating function (of type I-IV), the corresponding strict contactomorphism can be defined implicitly as well. This allows us to construct implicit contact integrators that preserve the contact structure exactly, while still leveraging the rich theory of symplectic integrators.

\begin{exmp}[Legendre-like Maps] A simple example of a symplectic map is the `Henon-like' map of \cite{turaev2002polynomial}, defined as:
\begin{align}
    H_{[f,\eta]}:\mqty(x\\p)\mapsto \mqty(p+\eta\\-x+\grad f(p))
\end{align}
where $f:\R^n\to\R$ is a smooth function and $\eta\in\R^n$ is a constant vector. The extension of this map to a strict contactomorphism on $J^1(\R^n)$ is given by:
\begin{equation}
    L_{[f,\eta,k]}:\mqty(x\\u\\p)\mapsto \mqty(p+\eta\\u -xp+f(p)+k\\-x+\grad f(p))
\end{equation}
where $k$ is an arbitrary constant. Observe that for $G=-xp+f(p)+k$ we have
\begin{align*}
    dG &= \grad f(p) dp - x dp - p dx \\
       &= (-x+\grad f(p)) dp - pdx \\
       &= H_{[f,\eta]}^*\qty(\bar{p} d\bar{x}) - p dx
\end{align*}
as required. The extended transformation resembles a `Legendre-like' transformation \cite{kevrekidis2024neural}.
\end{exmp}

When an exact or discrete generating function is available, the lifted map is a strict contactomorphism. If one instead reconstructs the generating function numerically, the resulting lift is generally only approximately strict, though it may still be useful in practice.

\subsubsection{Lifting ODE Integrators}

Let $\varphi:(x,u)\mapsto(\bar{x},\bar{u})$ be a diffeomorphism on $\R^n\times\R$. On a jet bundle, the variables $p$ can be interpreted as `derivatives' of $u$ with respect to the base $x$, and therefore the transformation of $p$ is determined by the chain rule from the transformation of $(x,u)$. More formally, the transformation $\:(x,u,p)\mapsto(\bar{x}=\chi(x,u),\bar{u}=\psi(x,u),\bar{p}=\pi(x,u,p))$ defined by
\begin{equation}
    \Phi(x,u,p) = (\bar{x}, \bar{u}, \bar{p})
\end{equation}
where $\bar{p}$ satisfies the following prolongation relationship \cite{olver1995equivalence}:
\begin{align}
    \bar{p} = \frac{d\bar{u}}{d\bar{x}}
    =(D_x\psi(x,u))(D_x\chi(x,u))^{-1} 
\end{align}
is a prolonged diffeomorphism on $J^1(\R^n)$. Here $D_x$ is the total derivative operator with respect to the base variables $x$, defined as:
\begin{align*}
    D_i=\partial_{x_i} + p_i \partial_u\qc D_x = \mqty(D_1\\\vdots\\D_n)\in(\R^n)^*
\end{align*} 
On $J^1(\R)$ for example, we have:
\begin{align*}
    \bar{p}=(D_x\psi)(D_x\chi)^{-1} = \frac{\partial_x \psi + p \partial_u \psi}{\partial_x \chi + p \partial_u \chi}
\end{align*}

Thus, one may `lift' an ODE integrator on $\R^n\times\R$ to a prolonged diffeomorphism on $J^1(\R^n)$ by first computing the flow on $(x,u)$ and then using the prolongation formula to define the transformation induced on the $p$ variables.

Now assume that $\varphi$ is defined by an ODE integrator, i.e. $\varphi = \Phi^h_F$ is the time-$h$ flow of a vector field $F$ on $\R^n\times\R$. Then, the transformation of $p$ can be computed by differentiating the flow $\Phi^h_F$ with respect to the initial conditions, which gives us the total derivatives $D_x\bar{u}$ and $D_x\bar{x}$ needed to compute $\bar{p}$ via the prolongation formula. Thus, if the integrator is implemented via differentiable programming techniques, the transformation of $p$ can be computed automatically by backpropagation through the integrator (i.e. automatic differentiation, giving us access to all of $\pdv{x}\chi,\pdv{u}\chi,\pdv{x}\psi,\pdv{u}\psi$ to the accuracy of the underlying differentiation pipeline). This allows us to construct explicit prolonged diffeomorphisms corresponding to any standard ODE integrator, such as explicit Euler methods, Runge-Kutta methods, and multistep methods, and therefore explicit contact integrators corresponding to these methods. Remarkably, this allows us to prolong even adaptive-step higher-order integrators (e.g. Dormand-Prince) to the jet space. Alternatively, an `adjoint method' can be used to compute the transformation of $p$ from the transformation of $(x,u)$, which is a standard technique in numerical analysis and machine learning for computing gradients of implicit functions. In this work, we use the automatic differentiation approach for simplicity, but the adjoint method is a viable alternative, depending on the specific application and computational resources available.

Finally, we note that the prolongation formula can be numerically ill-posed, since the inverse $(D_x\chi)^{-1}$ in the formula for $\bar{p}$ is not guaranteed to be well-behaved. This can lead to numerical instability in the computation of $\bar{p}$, especially for large time steps or for systems with stiff dynamics. In practice, one may need to use regularization techniques or adaptive time-stepping methods to mitigate this issue and ensure the stability of the integrator. This issue is not only a numerical artifact; it is intrinsic to contact dynamics, which can exhibit finite-time blow-up.

\begin{exmp}[Gradient Step] A simple example of a prolonged diffeomorphism is the `gradient step' map defined by:
\begin{equation}
    G_f:\mqty(x\\u)\mapsto \mqty(x\\u - f(x))
\end{equation}
where $f:\R^n\to\R$ is a smooth function. The extension of this map to a prolonged diffeomorphism on $J^1(\R^n)$ is given by:
\begin{equation}
    \tilde{G}_f:\mqty(x\\u\\p)\mapsto \mqty(x\\u - f(x)\\p - \grad f(x))
\end{equation}
which can be verified to satisfy the prolongation formula as follows:
\begin{align*}
    \bar{p} & = \frac{d\bar{u}}{d\bar{x}} \\
    & = \frac{d(u - f(x))}{dx} \\
    & = p - \grad f(x)
\end{align*}
as required. This transformation resembles a `gradient step' in optimization, and can be used to construct contact integrators for gradient-like systems, such as dissipative mechanical systems and optimization algorithms.
\end{exmp}

\subsection{Numerical Examples}

\subsubsection{The Damped Harmonic Oscillator}
The damped harmonic oscillator is a very simple example of a contact Hamiltonian system, given by the Hamiltonian
\begin{equation}
    H(x,u,p) = \frac{1}{2}p^2 + \frac{1}{2}x^2 + \gamma u
\end{equation}
where $x,p\in\R$ are the position and momentum variables, $u$ is a `friction' variable, which encodes the dissipation in the system, and $\gamma$ is a constant dissipation coefficient. The Hamiltonian satisfies the evolution equation $\dot{H} = -\gamma H$, which shows that the contact Hamiltonian along a trajectory decays exponentially at a rate determined by $\gamma$. This is an example that admits a closed-form solution we can compare our integrators to, and is a useful test case for demonstrating the performance of our integrators on a simple contact system.

Three reasonable splittings that arise naturally for this system are the following:
\begin{align}
    H & = \underbrace{\frac{1}{2}p^2}_{S} + \underbrace{\frac{1}{2}x^2 + \gamma u}_{P},\qq{Splitting 1}\\
    & = \underbrace{\frac{1}{2}p^2 + \frac{1}{2}x^2}_{S} + \underbrace{\gamma u}_{P},\qq{Splitting 2}\\
    &= \underbrace{\frac{1}{2}p^2}_{K} + \underbrace{\frac{1}{2}x^2}_{V} + \underbrace{\gamma u}_{D},\qq{Splitting 3}
\end{align}
with $S$ denoting the `strict' terms and $P$ denoting the `prolonged' terms from our two subalgebras, and the $K, V, D$ denoting the kinetic, potential, and dissipative terms respectively. We can then construct splitting integrators based on these splittings, and compare their performance to a standard non-structure-preserving integrator. While we expect all of our integrators to adhere to the contact structure and therefore to exhibit better long-term behavior than the non-structure-preserving integrator, the error behavior of the different splittings may differ (based on the bracket expansion of the modified Hamiltonian), and it is interesting to compare their performance in practice.

\begin{figure}
    \centering
    \includegraphics[width=0.8\textwidth]{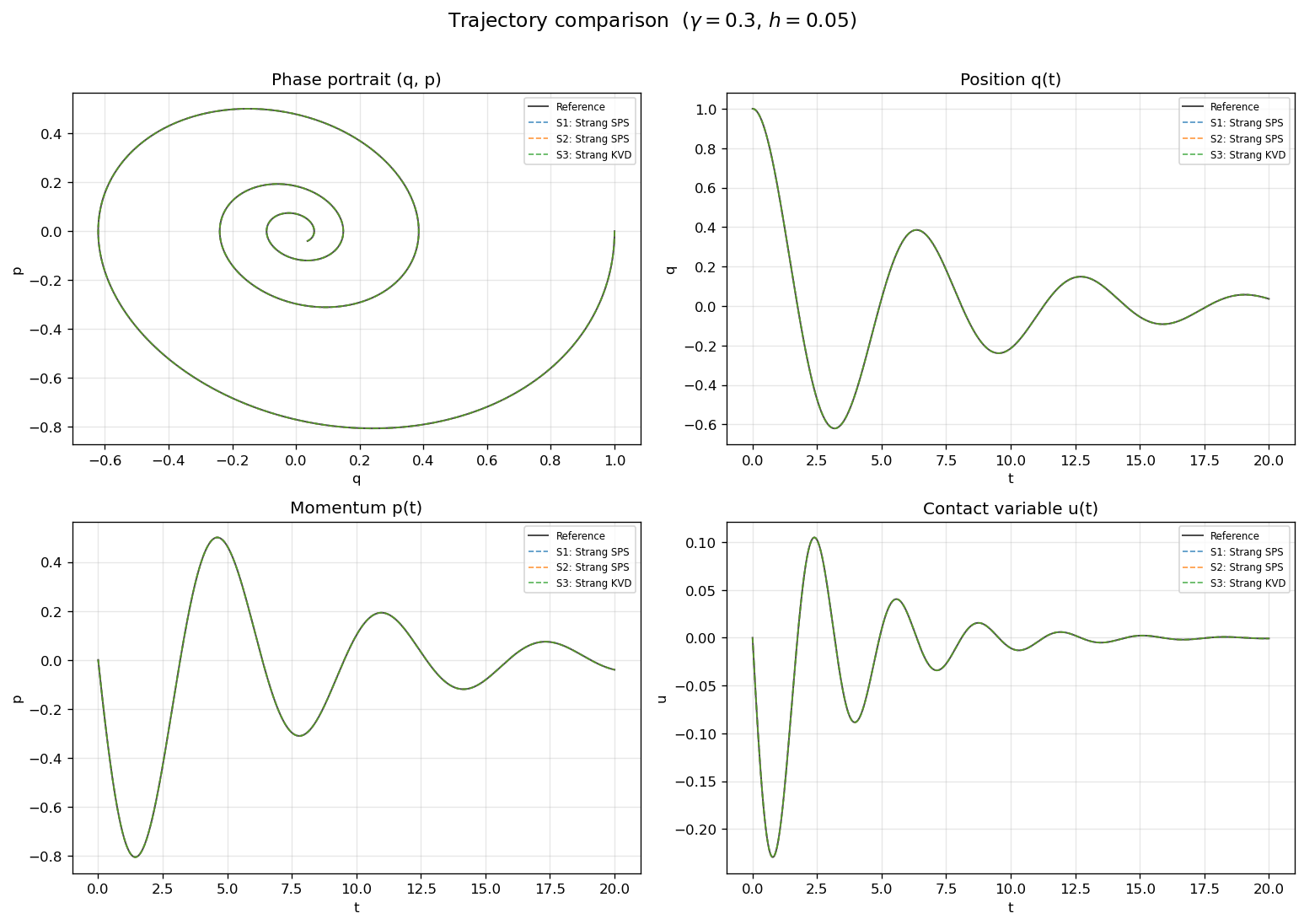}
    \caption{Reference trajectories of the damped harmonic oscillator for damping coefficient $\gamma=0.3$ and step-size $h=0.05$, for various splitting integrators (Strang SPS for Splittings 1 and 2, and Strang KVD for Splitting 3).}
    \label{fig:DHO_reference}
\end{figure}

In \cref{fig:DHO_reference} we visualize the reference trajectories of the damped harmonic oscillator for Strang splitting integrators based on the three different splittings; we can see that all three integrators exhibit the expected dissipative behavior, with the trajectories spiraling towards the origin in the phase space, with no difference being visible at this scale. In \cref{fig:DHO_RK4_comparison} we compare the performance of our splitting integrators to a standard non-structure-preserving integrator (RK4) by visualizing the error in the Hamiltonian $H$ over time. The RK4 and Yoshida-4 methods are both 4th-order integrators, with the latter being structure-preserving. In the middle and right panels we see the common oscillatory behavior of the error in the Hamiltonian and the conformal factor $\lambda$ for the structure-preserving integrators; however, the error in the Hamiltonian for the RK4 method steadily grows over time, which is expected for non-structure-preserving integrators. Moreover, the error in the conformal factor $\lambda$ for the RK4 method also grows over time, which indicates that the RK4 method is not accurately capturing the dissipative nature of the system, while our splitting integrators are able to capture this behavior in a stable manner.

For an extended comparison between different choices of splittings and integrators, see \cref{ap:Extended_Numerical}.

\begin{figure}
    \centering
    \includegraphics[width=0.99\textwidth]{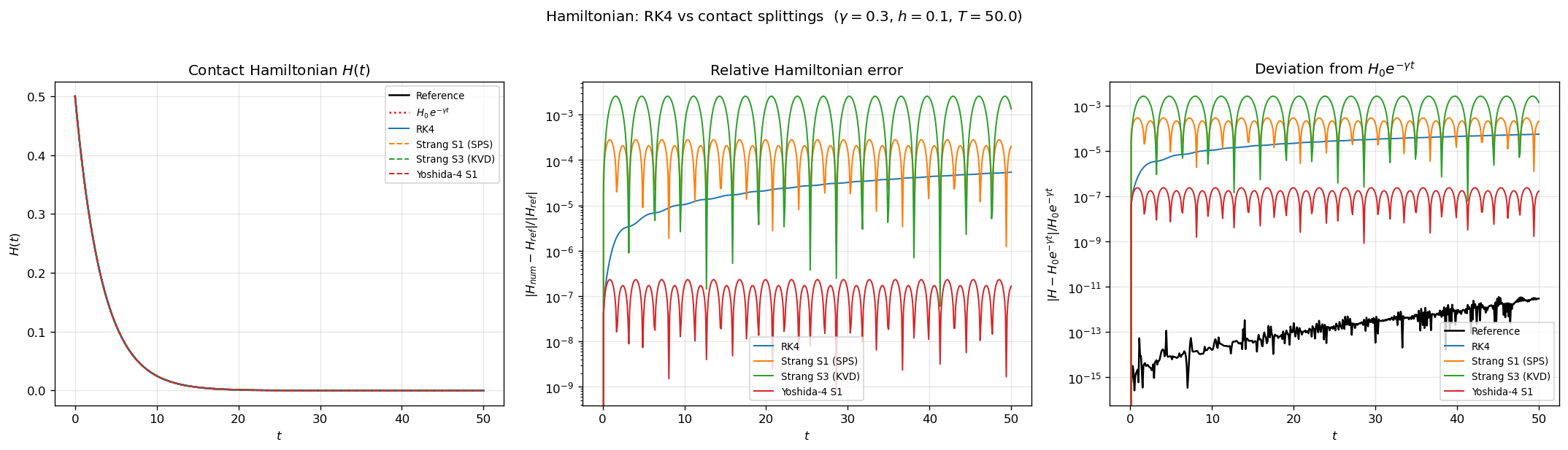}
    \caption{Comparison of contact splitting integrators (Strang SPS and Yoshida-4 for Splitting 1, Strang KVD for Splitting 3) to a standard non-structure-preserving integrator (RK4) for the damped harmonic oscillator with damping coefficient $\gamma=0.3$ and step-size $h=0.1$. (Left) Contact Hamiltonian decay over time. (Middle) Relative Hamiltonian Error over time. (Right) Relative error in the conformal factor $\lambda$ over time.}
    \label{fig:DHO_RK4_comparison}
\end{figure}

\subsubsection*{Contact Van der Pol Oscillator}

The contact Van der Pol oscillator (VdP) is a Lienard-type system that can be expressed as a contact Hamiltonian system. We consider the following forced time-dependent Hamiltonian:
\begin{equation}
    H(x,u,p) = pu - \epsilon(1-x^2)u - \frac{1}{2}x^2 + A\cos(\omega t)
\end{equation}
where $x,u,p\in\R$ are the position, energy, and momentum variables respectively, and $\epsilon$ is a parameter that controls the nonlinearity of the system. The parameter $\omega$ controls the frequency of the external forcing, and can generate classical or chaotic dynamics depending on its value \cite{pihajoki2015explicit,parlitz1987period}. Its integration using exact splittings is discussed in depth in \cite{zadra2021vdp}. We use it here for a different methodological reason: since the Hamiltonian is affine-in-$p$, it provides a clean test bed for the direct `base-integration $+$ prolongation' paradigm. In particular, it allows us to demonstrate that autodiff-based prolongation is computationally feasible in practice, that it compares favorably with a deeper splitting approach, and that it remains effective in a forced time-dependent regime that includes chaotic dynamics. We therefore compare a direct `base-integration $+$ prolongation' integrator for this system to the splitting approach of \cite{zadra2021vdp}. We make use of the autodiff-based prolongation approach for the first integrator, i.e.:
\begin{align}
    \bar{p}=\frac{\partial_x\psi+p\partial_u\psi}{\partial_x\chi+p\partial_u\chi}
\end{align}
where each of the four derivative terms on the right-hand side are computed by backpropagation through the ODE integrator for the base variables $(x,u)$, implemented in \texttt{JAX}. We use both a constant step-size RK4 integrator (4th-order) and an adaptive step-size DOP853 (8th-order) integrator for the base integration, and compare their performance to the Strang CBABC splitting integrator based on the splitting of the Hamiltonian into its kinetic, potential, and dissipative terms. We find that the direct `base-integration $+$ prolongation' integrator is able to capture the qualitative behavior of the system to high order. For the splitting approach, we use the splitting:
\begin{align*}
    H & = \underbrace{pu}_{C} + \underbrace{-\epsilon(1-x^2)u}_A -\underbrace{\frac{1}{2}x^2}_{B}
\end{align*}
with a Strang CBABC (2nd-order).

In \cref{fig:VdP_comparison} we visualize the trajectories obtained by each integrator in the phase space, and we see that all three integrators are able to capture the expected limit cycle behavior of the system, with the DOP853 integrator used as a reference for the base ODE. In \cref{fig:VdP_comparison_errors} we perform a convergence study in terms of the step-size and variable $\epsilon$ that controls the nonlinearity of the system for the RK4 and Strang CBABC integrators. We find that the advantage of the prolonged integrator is more apparent for larger values of $\epsilon$, as the step size varies, which is expected since the system becomes more nonlinear for larger values of $\epsilon$. Both methods are exact-contact, and therefore preserve the contact structure exactly, but the prolonged integrator is able to capture the dynamics more accurately for larger values of $\epsilon$ and larger step sizes. In this example, the main point is to show that autodiff-based prolongation is practically feasible, compares favorably with a deeper splitting, and remains effective in a forced regime with chaotic dynamics. For an extended numerical comparison between different choices of integrators and splittings, see \cref{ap:Extended_Numerical}.

\begin{figure}
    \centering
    \includegraphics[width=0.99\textwidth]{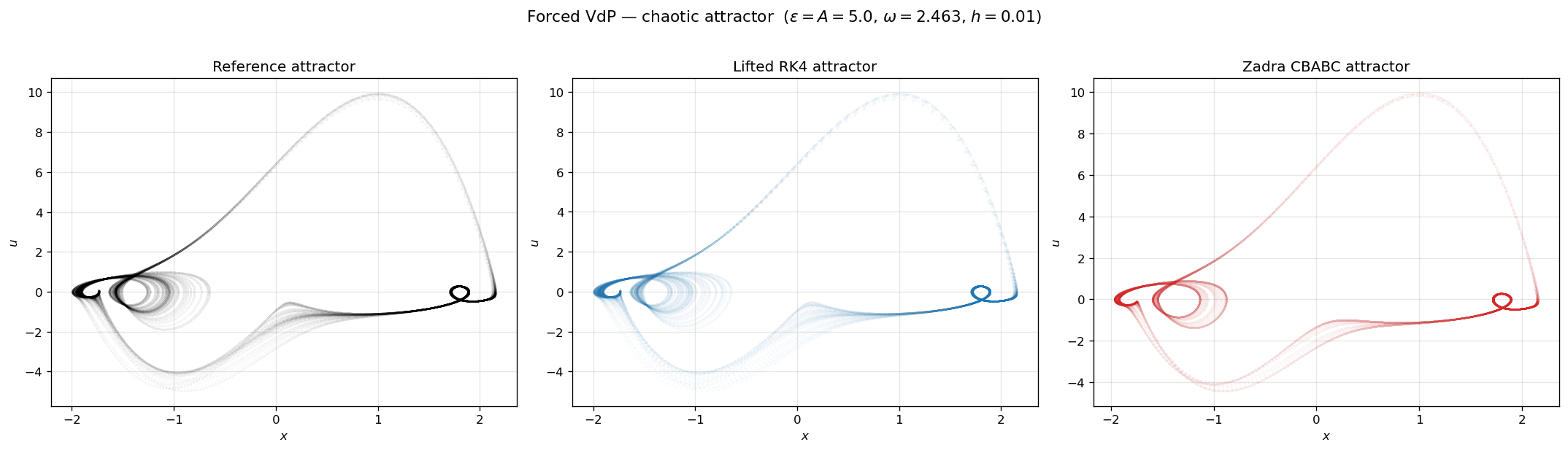}
    \caption{Comparison of a direct `base-integration + prolongation' integrator (RK4 and DOP853) to a splitting integrator (Strang CBABC) for the contact Van der Pol oscillator with $\epsilon=5.0$, $A=5.0$, $\omega=2.463$, and step-size $h=0.01$. (Left) Reference Trajectory (DOP853), (Middle) RK4, (Right) Strang CBABC.}
    \label{fig:VdP_comparison}
\end{figure}

\begin{figure}
    \centering
    \includegraphics[width=0.99\textwidth]{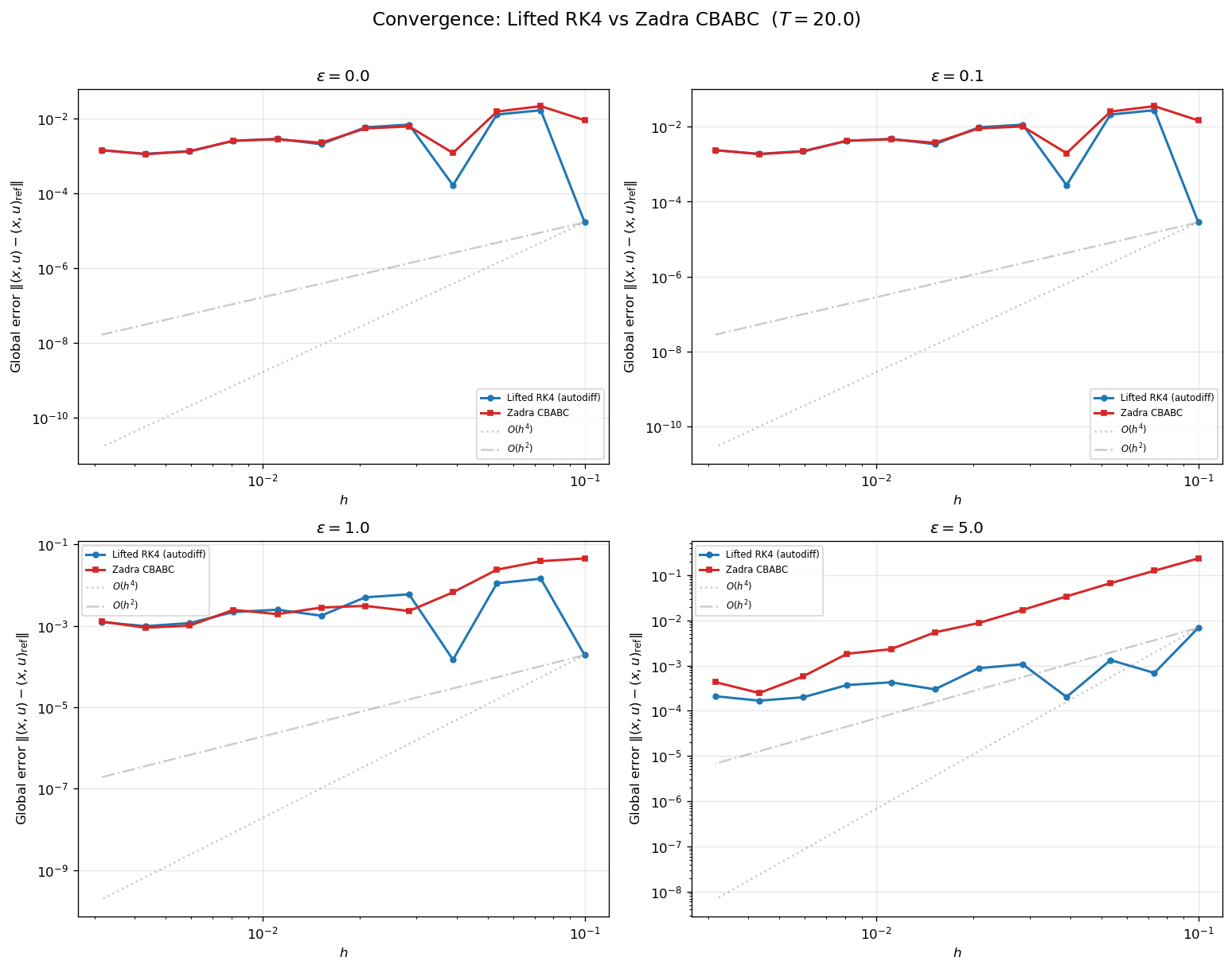}
    \caption{Comparison of a direct `base-integration + prolongation' integrator (RK4) to a splitting integrator (Strang CBABC) for the contact Van der Pol oscillator with $\epsilon=5.0$, $A=0.0$. The advantage of the prolonged integrator is more apparent for larger values of $\epsilon$, as the step size varies. Both methods are exact-contact.}
    \label{fig:VdP_comparison_errors}
\end{figure}

\subsubsection*{Nonlinear Dissipative Double-Well}

Finally, we consider a more complex system, given by a contact Hamiltonian of the form
\begin{equation}
    H(x,u,p)=\frac{1}{2}p^2 + (x^2-1)^2 + \sigma p^2u
\end{equation}
with a double well potential and a nonlinear dissipation term, which does not  belong directly to either of our subalgebras. The point of this example is to show that the nonlinear $p^2u$ term can still be handled by our splitting framework, including through a commutator gadget. This latter term can be approximated via splittings since:
\begin{align*}
    p^2u = \qty[-\frac{u^2}{2},p^2]
\end{align*}
with each term now belonging to one of our subalgebras. At the same time, this last term also admits an exact analytic splitting as a Bernoulli ODE in $p$, i.e.:
\begin{align*}
    x(t)&= x_0+2\sigma p_0 u_0 t \\
    p(t)&=\frac{p_0}{\sqrt(1+2\sigma p_0^2 t)}\\
    u(t)&=u_0\sqrt{1+2\sigma p_0^2 t}
\end{align*}

Overall, we consider three splittings for this system, which are the following: (A) a TV splitting where $T=\frac{1}{2}p^2+\sigma p^2 u$ yields a Bernoulli-type ODE and $V=(x^2-1)^2$ is a potential term (B) a CSC splitting where the $\sigma p^2 u$ term yields an exact Bernoulli-type ODE and the remainder is a symplectic Hamiltonian, and (C) a symplectic Hamiltonian for $\frac{1}{2}p^2+(x^2-1)^2$ plus a gadget splitting for the nonlinear $p^2u$ term. While each splitting is structure-preserving, the Bernoulli-ODEs admit closed-form approximations that can easily yield second-order integrators, while the gadget splitting requires a more complex composition of flows (the simplest iteration has global order $1/2$),
that is:
\begin{equation}
    e^{\sqrt{h}A}e^{\sqrt{h}B}e^{-\sqrt{h}A}e^{-\sqrt{h}B} = e^{h[A,B]} + O(h^{3/2})
\end{equation}
and obtaining higher-order integrators requires deeper compositions. This makes the example useful for separating the two roles of the framework. The TV and CSC splittings exploit additional structure specific to this Hamiltonian, and therefore yield the practically preferable structured integrators. By contrast, the gadget splitting treats the nonlinear $p^2u$ term through the generic commutator route, showing mixed $(p,u)$ terms can still be handled when no more direct structured decomposition is available.

For the three splittings, we perform a numerical convergence comparison in terms of the step size $h$ in \cref{fig:Chaos_comparison} for different values of $\sigma$, where we find that the TV and CSC splittings perform similarly, while the naive gadget splitting performs significantly worse. 

\begin{figure}
    \centering
    \includegraphics[width=0.99\textwidth]{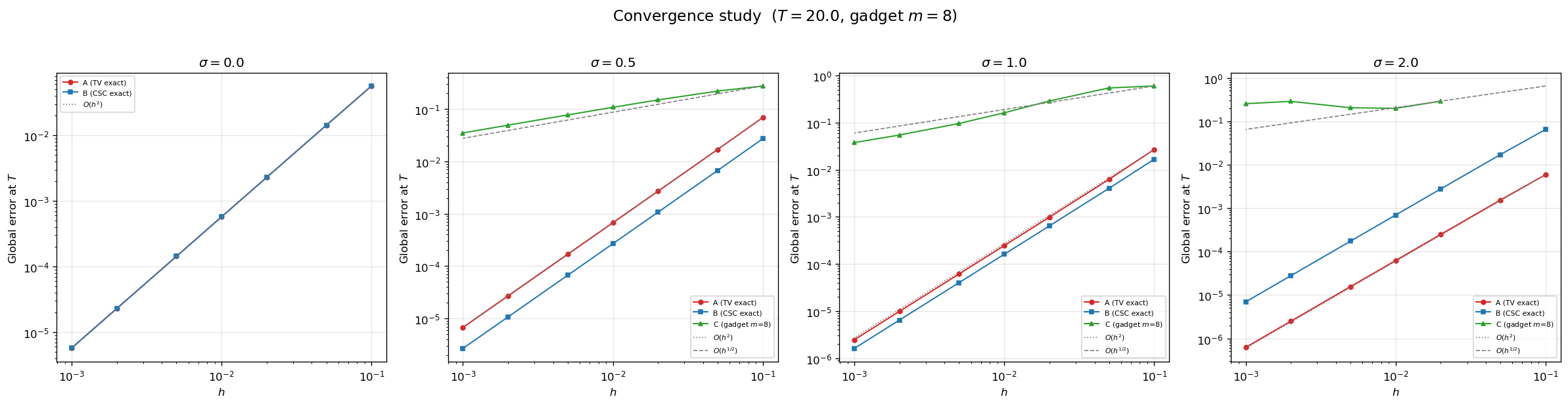}
    \caption{Convergence analysis for the nonlinear dissipative double-well system with varying step size $h$ and nonlinear parameter $\sigma$.}
    \label{fig:Chaos_comparison}
\end{figure}

In \cref{fig:Chaos_sigma_comparison}, we compare attracting sets for a fixed step size $h=0.01$ and varying values of $\sigma$. Careful examination of the figures demonstrates a qualitatively different behavior of the gadget-splitting, compared to the TV and CSC splittings, which are more similar to each other. Extended numerical results are presented in \cref{ap:Extended_Numerical}, where we additionally demonstrate higher-order gadgets for approximating the flow of the commutator.

\begin{figure}
    \centering
    \includegraphics[width=0.99\textwidth]{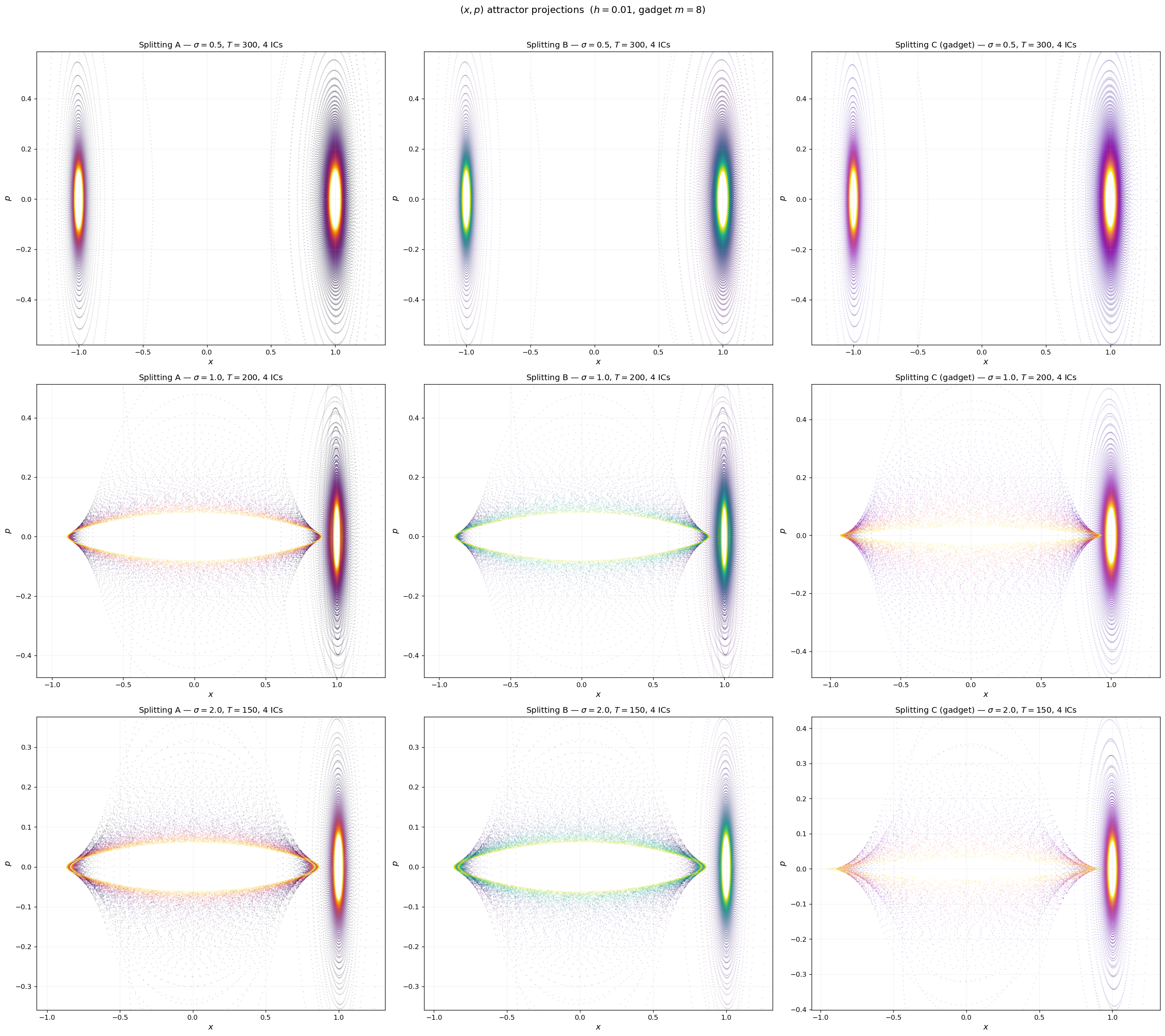}
    \caption{Comparison of attracting sets obtained by integrating four distinct initial conditions, for the three different splittings of the nonlinear system. (Left) TV splitting, (Middle) CSC splitting, (Right) Gadget splitting.}
    \label{fig:Chaos_sigma_comparison}
\end{figure}

For an extended numerical comparison between different choices of splittings and integrators, see \cref{ap:Extended_Numerical}.

\section{Discussion}
\label{sec:Discussion}

In this work, we have developed a framework for constructing structure-preserving integrators for local contact Hamiltonian dynamics, with global statements proved on the jet space $J^1(\R^n)$, by combining two tractable classes of exact contact subflows: strict contactomorphisms and prolonged diffeomorphisms. The main theoretical ingredient is the Lie-algebra density result for the corresponding strict and prolonged Hamiltonians. Combined with standard BCH-based product-formula constructions, this yields a universal splitting framework on $J^1(\R^n)$ together with approximation bounds and asymptotic error estimates. The lifting constructions in the numerical section then show that these abstract building blocks can be realized concretely from symplectic integrators on $T^*\R^n$ and ODE integrators on $\R^n\times\R$, and the examples illustrate this realization on representative low-dimensional systems.

The broader significance of this work is therefore methodological as much as it is application-specific. Rather than constructing contact integrators on a model-by-model basis from the relatively small collection of contact Hamiltonians with closed-form subflows, the present framework reduces the problem to two exact-contact building blocks for which mature lower-dimensional numerical toolkits already exist. The polynomial approximation argument supplies a formal universality guarantee on compact sets, but it should not be read as the main practical pipeline in high dimension. Practically, the main advantage of choosing such a large pair of subalgebras is that many Hamiltonians already split into strict and prolonged pieces, or require only a small number of commutator corrections, before any approximation step is invoked. In this sense, the framework provides a transfer principle from established numerical methods to the contact setting: advances in symplectic integration can be imported through strict lifts, while advances in standard ODE solvers, variational equations, and automatic differentiation can be imported through prolonged lifts. The low-dimensional examples in this paper should therefore be viewed as proofs of concept for a reusable construction strategy, rather than as the natural limit of the framework's applicability.

This viewpoint suggests broad applicability whenever contact geometry provides the appropriate state-space description of dissipative or non-conservative dynamics. In scientific computing, it offers a systematic route to contact-preserving time steppers for dissipative mechanical models, thermodynamic systems, and related non-conservative evolutions \cite{simoes2020contact,grmela2014contact_thermo}. In optimization and machine learning, the same strict/prolonged decomposition suggests a principled way to build contact-preserving flow maps from lower-dimensional components whose numerical realization is already well understood. More generally, the framework points toward extending the toolbox of geometric integration beyond the conservative symplectic setting and into a substantially wider class of structured dynamical systems.

Several directions remain open. On the analytical side, the Lie-density result suggests possible connections with controllability questions and with normal-form constructions for contact dynamics. On the geometric side, it would be interesting to extend the framework beyond jets and local Darboux coordinates to broader classes of contact manifolds, such as cosphere bundles and projectivized cotangent bundles; on manifolds with nontrivial topology, one would first need a principled way to patch chart-local constructions into a global splitting framework. On the numerical side, the examples also indicate that practical performance depends on the chosen representation and splitting, so systematic strategies for selecting generators and compositions for specific problems remain an important topic for further study \cite{mclachlan2002splitting}.

\section{Acknowledgements}
I am grateful to my PhD advisors at Johns Hopkins University, Mauro Maggioni and Soledad Villar, for their continued support and guidance throughout this work. I would also like to thank Joshua Burby and Peter Olver for many helpful discussions during the early stages of this work, which were instrumental in shaping the ideas developed here.

\newpage 

\bibliography{main}

\begin{thebibliography}{10}

\bibitem{kholodenko2013applications_contact}
Arkady~L Kholodenko.
\newblock {\em Applications of contact geometry and topology in physics}.
\newblock World Scientific, 2013.

\bibitem{arnold2013mechanics}
Vladimir~Igorevich Arnol'd.
\newblock {\em Mathematical methods of classical mechanics}, volume~60.
\newblock Springer Science \& Business Media, 2013.

\bibitem{mrugala1991contact_thermo}
Ryszard Mrugala, James~D Nulton, J~Christian Sch{\"o}n, and Peter Salamon.
\newblock Contact structure in thermodynamic theory.
\newblock {\em Reports on mathematical physics}, 29(1):109--121, 1991.

\bibitem{grmela2014contact_thermo}
Miroslav Grmela.
\newblock Contact geometry of mesoscopic thermodynamics and dynamics.
\newblock {\em Entropy}, 16(3):1652--1686, 2014.

\bibitem{lie_contact}
S.~Lie and G.~Scheffers.
\newblock {\em Geometrie der Beriihrungstransformationen}.
\newblock B.G. Teubner, Leipzig, 1896.

\bibitem{olver1995equivalence}
Peter~J Olver.
\newblock {\em Equivalence, invariants and symmetry}.
\newblock Cambridge University Press, 1995.

\bibitem{geiges2001brief}
Hansj{\"o}rg Geiges.
\newblock A brief history of contact geometry and topology.
\newblock {\em Expositiones Mathematicae}, 19(1):25--53, 2001.

\bibitem{geiges2006contact}
Hansj{\"o}rg Geiges.
\newblock Contact geometry.
\newblock In {\em Handbook of differential geometry}, volume~2, pages 315--382.
  Elsevier, 2006.

\bibitem{massot2014topological}
Patrick Massot.
\newblock Topological methods in 3-dimensional contact geometry.
\newblock {\em Contact and symplectic topology}, 26:27--83, 2014.

\bibitem{de2019contact}
Manuel de~Le{\'o}n and Manuel Lainz~Valc{\'a}zar.
\newblock Contact hamiltonian systems.
\newblock {\em Journal of Mathematical Physics}, 60(10), 2019.

\bibitem{simoes2020contact}
Alexandre~Anahory Simoes, Manuel de~Le{\'o}n, Manuel~Lainz Valc{\'a}zar, and
  David~Mart{\'\i}n de~Diego.
\newblock Contact geometry for simple thermodynamical systems with friction.
\newblock {\em Proceedings of the Royal Society A}, 476(2241):20200244, 2020.

\bibitem{de2023time}
Manuel de~Le{\'o}n, Jordi Gaset, Xavier Gr{\`a}cia, Miguel~C Mu{\~n}oz-Lecanda,
  and Xavier Rivas.
\newblock Time-dependent contact mechanics.
\newblock {\em Monatshefte f{\"u}r Mathematik}, 201(4):1149--1183, 2023.

\bibitem{gaset2020contact}
Jordi Gaset, Xavier Gracia, Miguel~C Mu{\~n}oz-Lecanda, Xavier Rivas, and
  Narciso Rom{\'a}n-Roy.
\newblock A contact geometry framework for field theories with dissipation.
\newblock {\em Annals of Physics}, 414:168092, 2020.

\bibitem{feng2010contact}
Kang Feng and Mengzhao Qin.
\newblock Contact algorithms for contact dynamical systems.
\newblock In {\em Symplectic Geometric Algorithms for Hamiltonian Systems},
  pages 477--497. Springer, 2010.

\bibitem{feng2010symplectic}
Kang Feng and Mengzhao Qin.
\newblock {\em Symplectic geometric algorithms for Hamiltonian systems}, volume
  449.
\newblock Springer, 2010.

\bibitem{araujo2025jacobi}
Ad{\'e}rito Ara{\'u}jo, Gon{\c{c}}alo~Inoc{\^e}ncio Oliveira, and Jo{\~a}o~Nuno
  Mestre.
\newblock Jacobi hamiltonian integrators.
\newblock {\em arXiv preprint arXiv:2507.18573}, 2025.

\bibitem{bravetti2020numerical}
Alessandro Bravetti, Marcello Seri, Mats Vermeeren, and Federico Zadra.
\newblock Numerical integration in celestial mechanics: a case for contact
  geometry.
\newblock {\em Celestial Mechanics and Dynamical Astronomy}, 132(1):7, 2020.

\bibitem{vermeeren2019contact}
Mats Vermeeren, Alessandro Bravetti, and Marcello Seri.
\newblock Contact variational integrators.
\newblock {\em Journal of Physics A: Mathematical and Theoretical},
  52(44):445206, 2019.

\bibitem{guenther1996herglotz}
Ronald~B Guenther, Hans Schwerdtfeger, Gustav Herglotz, CM~Guenther, and
  JA~Gottsch.
\newblock {\em The Herglotz lectures on contact transformations and Hamiltonian
  systems}.
\newblock Juliusz Schauder Center for Nonlinear Studies. Nicholas Copernicus
  University, 1996.

\bibitem{zadra2023topics}
Federico Zadra.
\newblock Topics in contact hamiltonian systems: analytical and numerical
  perspectives.
\newblock {\em PhD thesis}, 2023.

\bibitem{zadra2021vdp}
Federico Zadra, Alessandro Bravetti, and Marcello Seri.
\newblock Geometric numerical integration of li{\'e}nard systems via a contact
  hamiltonian approach.
\newblock {\em Mathematics}, 9(16):1960, 2021.

\bibitem{mclachlan2002splitting}
Robert~I McLachlan and G~Reinout~W Quispel.
\newblock Splitting methods.
\newblock {\em Acta Numerica}, 11:341--434, 2002.

\bibitem{francca2021dissipative}
Guilherme Fran{\c{c}}a, Michael~I Jordan, and Ren{\'e} Vidal.
\newblock On dissipative symplectic integration with applications to
  gradient-based optimization.
\newblock {\em Journal of Statistical Mechanics: Theory and Experiment},
  2021(4):043402, 2021.

\bibitem{bravetti2023bregman}
Alessandro Bravetti, Maria~L Daza-Torres, Hugo Flores-Arguedas, and Michael
  Betancourt.
\newblock Bregman dynamics, contact transformations and convex optimization.
\newblock {\em Information Geometry}, 6(1):355--377, 2023.

\bibitem{testa2025geometric}
Andrea Testa, S{\o}ren Hauberg, Tamim Asfour, and Leonel Rozo.
\newblock Geometric contact flows: Contactomorphisms for dynamics and control.
\newblock {\em arXiv preprint arXiv:2506.17868}, 2025.

\bibitem{krieg_michorl1997convenient}
Andreas Kriegl and Peter~W Michor.
\newblock {\em The convenient setting of global analysis}, volume~53.
\newblock American Mathematical Soc., 1997.

\bibitem{banyaga2013structure}
Augustin Banyaga.
\newblock {\em The structure of classical diffeomorphism groups}, volume 400.
\newblock Springer Science \& Business Media, 2013.

\bibitem{honda2019notes}
Ko~Honda.
\newblock Notes for math 599: contact geometry.
\newblock {\em Lecture Notes available at http://www-bcf. usc. edu/\~{}
  khonda/math599/notes. pdf}, 2019.

\bibitem{hairer2006geometric}
Ernst Hairer, Christian Lubich, and Gerhard Wanner.
\newblock {\em Geometric numerical integration}, volume~31 of {\em Springer
  Series in Computational Mathematics}.
\newblock Springer-Verlag, Berlin, second edition, 2006.
\newblock Structure-preserving algorithms for ordinary differential equations.

\bibitem{narasimhan1985analysis}
Raghavan Narasimhan.
\newblock {\em Analysis on real and complex manifolds}, volume~35.
\newblock Elsevier, 1985.

\bibitem{devore1993constructive}
Ronald~A DeVore and George~G Lorentz.
\newblock {\em Constructive approximation}, volume 303.
\newblock Springer Science \& Business Media, 1993.

\bibitem{suzuki1990fractal}
Masuo Suzuki.
\newblock Fractal decomposition of exponential operators with applications to
  many-body theories and monte carlo simulations.
\newblock {\em Physics Letters A}, 146(6):319--323, 1990.

\bibitem{yoshida1990construction}
Haruo Yoshida.
\newblock Construction of higher order symplectic integrators.
\newblock {\em Physics letters A}, 150(5-7):262--268, 1990.

\bibitem{turaev2002polynomial}
Dmitry Turaev.
\newblock Polynomial approximations of symplectic dynamics and richness of
  chaos in non-hyperbolic area-preserving maps.
\newblock {\em Nonlinearity}, 16(1):123, 2002.

\bibitem{kevrekidis2024neural}
George~A Kevrekidis, Daniel~A Serino, Alexander Kaltenborn, J~Tinka Gammel,
  Joshua~W Burby, and Marc~L Klasky.
\newblock Neural network representations of multiphase equations of state.
\newblock {\em arXiv preprint arXiv:2406.19957}, 2024.

\bibitem{pihajoki2015explicit}
Pauli Pihajoki.
\newblock Explicit methods in extended phase space for inseparable hamiltonian
  problems.
\newblock {\em Celestial Mechanics and Dynamical Astronomy}, 121(3):211--231,
  2015.

\bibitem{parlitz1987period}
Ulrich Parlitz and Werner Lauterborn.
\newblock Period-doubling cascades and devil’s staircases of the driven van
  der pol oscillator.
\newblock {\em Physical Review A}, 36(3):1428, 1987.

\bibitem{geiges2008introduction}
Hansj{\"o}rg Geiges.
\newblock {\em An introduction to contact topology}, volume 109.
\newblock Cambridge University Press, 2008.

\bibitem{apostol1973mathematical}
Tom~M Apostol.
\newblock {\em Mathematical analysis: A modern approach to advanced calculus}.
\newblock Addison-Wesley, Reading, Mass, 2nd edition, 1973.

\bibitem{hartman2002ordinary}
Philip Hartman.
\newblock {\em Ordinary differential equations}.
\newblock SIAM, 2002.

\bibitem{abraham2012manifolds}
Ralph Abraham, Jerrold~E Marsden, and Tudor Ratiu.
\newblock {\em Manifolds, tensor analysis, and applications}.
\newblock Springer Science \& Business Media, 2012.

\end{thebibliography}
\bibliographystyle{unsrt}

\newpage

\appendix

\section{Auxiliary Results}

\subsection{Coordinate-Free Contact Hamiltonian Formalism}
\label{ap:Coordinate_Free_Contact}

Let $(\mcM,\xi=\ker\alpha)$ be a contact manifold, and fix a contact form $\alpha$ defining $\xi$.
The \textbf{Reeb vector field} $R_\alpha\in\mfX(\mcM)$ is the unique vector field satisfying
\begin{equation}
    \alpha(R_\alpha)=1,
    \qquad
    \iota_{R_\alpha}d\alpha=0.
\end{equation}

A vector field $X\in\mfX(\mcM)$ is called a \textbf{contact vector field} if
\begin{equation}
    \mathcal{L}_X\alpha=\zeta_X\alpha
\end{equation}
for some smooth function $\zeta_X$ on $\mcM$.

Once the contact form $\alpha$ is fixed, every smooth function $H\in C^\infty(\mcM)$ determines a unique \textbf{contact Hamiltonian vector field} $X_H\in\mfX(\mcM)$ by
\begin{equation}
    \alpha(X_H)=-H,
    \qquad
    \iota_{X_H}d\alpha=dH-(R_\alpha H)\alpha.
\end{equation}
Conversely, if $X$ is a contact vector field, then its Hamiltonian with respect to $\alpha$ is
\begin{equation}
    H=-\alpha(X).
\end{equation}

By Cartan's formula,
\begin{equation}
    \mathcal{L}_{X_H}\alpha
    =
    d(\alpha(X_H))+\iota_{X_H}d\alpha
    =
    -dH+dH-(R_\alpha H)\alpha
    =
    -(R_\alpha H)\alpha,
\end{equation}
so $X_H$ is indeed a contact vector field.

In Darboux coordinates $(x^i,u,p_i)$ with
\begin{equation}
    \alpha=du-p_i dx^i,
\end{equation}
the Reeb field is $R_\alpha=\pdv{u}$, and the definition above reduces to
\begin{equation}
X_H
=
\frac{\partial H}{\partial p_i}\frac{\partial}{\partial x^i}
-
\qty(\frac{\partial H}{\partial x^i}+p_i\frac{\partial H}{\partial u})\frac{\partial}{\partial p_i}
+
\qty(p_i\frac{\partial H}{\partial p_i}-H)\pdv{u},
\end{equation}
which is the coordinate formula used in the main text.

Thus, the coordinate expressions in Darboux charts are local representatives of a coordinate-free construction defined globally once a contact form $\alpha$ is chosen. Note, however, that the Reeb field and the Hamiltonian $H=-\alpha(X_H)$ depend on the choice of contact form, whereas the underlying contact structure $\xi$ and the notion of a contact vector field do not. This coordinate-free formalism clarifies the intrinsic objects involved, but it does not by itself furnish a global analogue of the splitting construction on an arbitrary contact manifold. For that, one would additionally need globally compatible generators and a way to patch the chart-local constructions used in the main text.

\subsection{Proofs of Lemmas 2.4 and 2.5}

We will work directly with the coordinate form of the Hamiltonians.

\begin{proof}[Proof of Lemma 2.4]
    Let $H(x,u,p)=K(x,p)$, where $K\in C^\infty(T^*\R^n)$. Then the contact Hamiltonian vector field generated by $H$ is given by
    \begin{equation}
        \dot{x}^i=\pdv{K}{p_i},\qquad
        \dot{p}_i=-\pdv{K}{x^i},\qquad
        \dot{u}=p_i\frac{\partial K}{\partial p_i}-K.
    \end{equation}
    Thus the $(x,p)$ variables evolve according to the symplectic Hamiltonian flow generated by $K$, while the $u$ variable evolves according to an auxiliary ODE along that trajectory. Denoting the symplectic Hamiltonian flow of $K$ on $T^*\R^n$ by $\Phi_K^t$, we therefore have
    \begin{align*}
        \pi_{(x,p)}\circ\Phi_H^t(x,u,p)=\Phi_K^t\circ \pi_{(x,p)}(x,u,p)
    \end{align*}
    where $\pi_{(x,p)}:J^1(\R^n)\to T^*\R^n$ is the canonical projection on the $(x,p)$ variables. We have therefore shown that the projection of the contact flow generated by $H$ onto the $(x,p)$ variables coincides with the symplectic Hamiltonian flow generated by $K$. 

    We further claim that the flow $\Phi_H^t$ is strict, i.e. each time-$t$ map preserves the contact form $\alpha$. This follows directly from the formula
    \begin{align*}
        \mcL_{X_H}\alpha = -\pdv{H}{u}\alpha = 0
    \end{align*}
    since $H$ is independent of $u$. Thus, the flow generated by $H$ consists of strict contactomorphisms.

    Conversely, given a symplectic Hamiltonian flow $\Phi_K^t$ on $T^*\R^n$ generated by $K$, define the contact Hamiltonian $H(x,u,p)=K(x,p)$. The same computation shows that its contact flow $\Phi_H^t$ satisfies
    \begin{align*}
        \pi_{(x,p)}\circ\Phi_H^t=\Phi_K^t\circ\pi_{(x,p)}
    \end{align*}
    and consists of strict contactomorphisms. This is precisely the asserted lift of the symplectic Hamiltonian flow generated by $K$.
\end{proof}

For a standard reference on this type of correspondence (referred to more generally as ``symplectization of a contact manifold" or ``contactization of a symplectic manifold"), see \cite{geiges2008introduction,banyaga2013structure,guenther1996herglotz}.

\begin{proof}[Proof of Lemma 2.5]
    Suppose that $H(x,u,p)=f(x,u)+g^i(x,u)p_i$ is a contact Hamiltonian affine in $p$. Then the induced evolution equations on the $x$ and $u$ variables are given by
    \begin{equation}
        \dot{x}^i=g^i(x,u),\qquad
        \dot{u}=-f(x,u).
    \end{equation}
    Thus, the $(x,u)$-evolution is autonomous and independent of the momentum variable $p$, and, subject to the well-posedness of the ODE, it generates a flow $\varphi^t$ on $\R^n\times\R$ with infinitesimal generator
    \begin{equation*}
        Y = g^i(x,u)\pdv{x^i} - f(x,u)\pdv{u}.
    \end{equation*}
    We therefore have that
    \begin{align*}
        \pi_{(x,u)}\circ\Phi_H^t(x,u,p)=\varphi^t\circ \pi_{(x,u)}(x,u,p).
    \end{align*}

    To identify the full lifted flow, recall that the first prolongation of a vector field
    \begin{equation*}
        Y = \xi^i(x,u)\pdv{x^i} + \phi(x,u)\pdv{u}
    \end{equation*}
    on $\R^n\times\R$ is the vector field
    \begin{equation*}
        Y^{(1)} = \xi^i\pdv{x^i} + \phi\pdv{u} + \qty(D_i\phi - p_j D_i\xi^j)\pdv{p_i},
        \qquad D_i=\partial_{x^i}+p_i\partial_u.
    \end{equation*}
    Substituting $\xi^i=g^i$ and $\phi=-f$ gives
    \begin{align*}
        Y^{(1)}
        &= g^i\pdv{x^i} - f\pdv{u} + \qty(-D_i f - p_j D_i g^j)\pdv{p_i} \\
        &= g^i\pdv{x^i} - f\pdv{u}
        - \qty(\pdv{f}{x^i} + p_i\pdv{f}{u} + p_j\pdv{g^j}{x^i} + p_i p_j\pdv{g^j}{u})\pdv{p_i}.
    \end{align*}

    On the other hand, the contact Hamiltonian vector field of $H=f+g^i p_i$ is
    \begin{align*}
        X_H
        &= \pdv{H}{p_i}\pdv{x^i}
        - \qty(\pdv{H}{x^i}+p_i\pdv{H}{u})\pdv{p_i}
        + \qty(p_i\pdv{H}{p_i}-H)\pdv{u} \\
        &= g^i\pdv{x^i}
        - \qty(\pdv{f}{x^i} + p_j\pdv{g^j}{x^i} + p_i\pdv{f}{u} + p_i p_j\pdv{g^j}{u})\pdv{p_i}
        - f\pdv{u}.
    \end{align*}
    Hence $X_H=Y^{(1)}$. Therefore the contact flow $\Phi_H^t$ is exactly the first prolongation of the base flow $\varphi^t$.

    Conversely, given the autonomous base vector field
    \begin{equation*}
        Y = g^i(x,u)\pdv{x^i} - f(x,u)\pdv{u},
    \end{equation*}
    define the contact Hamiltonian
    \begin{equation*}
        H(x,u,p)=f(x,u)+g^i(x,u)p_i.
    \end{equation*}
    The same calculation yields $X_H=Y^{(1)}$, so the first prolongation of the base flow generated by $Y$ is precisely the contact flow generated by $H$.

    For a standard reference on the prolongation construction, see \cite{olver1995equivalence}. Note that this is a construction that is specific to a Jet-space setting, due to the prolongation formula for the $p$ variable.
\end{proof}

\subsection{The Poisson Bracket}
\label{ap:Poisson_Bracket}

The Poisson bracket is a fundamental object in symplectic geometry and Hamiltonian mechanics. In this work, it appears directly when expressing the contact-Jacobi bracket in Darboux coordinates. Here, we provide a brief overview of the Poisson bracket and its key properties.

\begin{defn}
Let $(\mcM,\omega)$ be a symplectic manifold, where $\omega$ is a symplectic form on $\mcM$. To match the contact-Hamiltonian sign convention used in the main text, for any two smooth functions $f,g \in C^{\infty}(\mcM)$, the \emph{Poisson bracket} $\{f,g\}$ is defined by
\begin{equation}
    \{f,g\} = -\omega(X_f, X_g),
\end{equation}
where $X_f$ and $X_g$ are the Hamiltonian vector fields associated with $f$ and $g$, respectively defined by the the equations:
\begin{equation}
    \iota_{X_f} \omega = df\qc \iota_{X_g} \omega = dg.
\end{equation}
\end{defn}

In local Darboux coordinates $(x^i, p_i)$ on $\mcM$, the Poisson bracket can be expressed as
\begin{equation}
    \{f,g\} = \sum_{i}^n \qty(\pdv{f}{p_i} \pdv{g}{x^i} - \pdv{f}{x^i} \pdv{g}{p_i}),
\end{equation}
where $n$ is half the dimension of the symplectic manifold $\mcM$.

In our work, we primarily consider the standard symplectic structure on the cotangent bundle $T^*\R^n$ with coordinates $(x^i, p_i)$, where $x^i$ are position coordinates and $p_i$ are momentum coordinates. In this setting, the Poisson bracket interacts with polynomials in the momentum variables in a specific way, as described in the following proposition.

\begin{prop}\label{prop:poisson_bracket_properties}
Let $f,g \in C^{\infty}(T^*\R^n)$ polynomial in the momentum variables $p$ of degree at most $m$ and $k$, respectively. Then, the Poisson bracket $\{f,g\}$ is polynomial in $p$ of degree at most $m+k-1$.
\begin{proof}
The Poisson bracket is defined as
\begin{equation*}
    \{f,g\} = \sum_{i=1}^n \qty( \pdv{f}{p_i} \pdv{g}{x^i} - \pdv{f}{x^i} \pdv{g}{p_i} ).
\end{equation*}
Since $f$ is polynomial in $p$ of degree at most $m$, the partial derivative $\pdv{f}{p_i}$ is polynomial in $p$ of degree at most $m-1$. Similarly, since $g$ is polynomial in $p$ of degree at most $k$, the partial derivative $\pdv{g}{p_i}$ is polynomial in $p$ of degree at most $k-1$. Therefore, each term in the sum defining the Poisson bracket involves either $\pdv{f}{p_i}$ or $\pdv{g}{p_i}$, which reduces the degree in $p$ by one. Consequently, the Poisson bracket $\{f,g\}$ is polynomial in $p$ of degree at most $m+k-1$.
\end{proof}
\end{prop}

\subsection{The Euler Operator}
\label{app:Euler_Op}

The Euler operator\footnote{This particular term is used to refer to various different objects in the literature. The naming here is consistent with \cite{arnold2013mechanics}} appears in the explicit formula for the contact-Jacobi bracket. In this short section we state some of its properties that appear in the proofs of some of the main results.

\begin{defn}
The \emph{Euler operator} $E:C^{\infty}(J^1(\R^n)) \to C^{\infty}(J^1(\R^n))$ is defined by
\begin{equation}\label{eqn:euler_op}
    E:f \mapsto f - p \pdv{f}{p},
\end{equation}
where $f$ is a smooth scalar function on the first jet space $J^1(\R^n)$ with Darboux coordinates $(x,p,u)$.
\end{defn}

The Euler operator provides a measure of homogeneity of a function with respect to the momentum variables $p$. In particular, it annihilates functions that are homogeneous of degree one in $p$.

\begin{prop}\label{prop:euler_op_properties}
The Euler operator $E$ satisfies the following properties:
\begin{enumerate}
    \item $E$ is a linear operator, i.e., for any $f,g \in C^{\infty}(J^1(\R^n))$ and scalars $a,b \in \R$, we have
    \begin{equation}
        E[af + bg] = aE[f] + bE[g].
    \end{equation}
    \item For any $f \in C^{\infty}(J^1(\R^n))$, if $f$ is homogeneous of degree $k$ in the momentum variables $p$, then
    \begin{equation}
        E[f] = (1 - k)f.
    \end{equation}
    \item If $f$ is polynomial in the momentum variables $p$, then
    \begin{equation}
        E[f] = \sum_{k=0}^{d} (1 - k) f_k
    \end{equation}
    where $f_k$ is the homogeneous component of degree $k$ in $p$, and $d$ is the highest degree of $p$ in $f$.
\end{enumerate}
\begin{proof}
    For (1), the linearity of $E$ follows directly from its definition:
    \begin{align*}
        E[af + bg] &= (af + bg) - p \pdv{(af + bg)}{p} \\
        &= af + bg - p \qty( a\pdv{f}{p} + b\pdv{g}{p}) \\
        &= a\qty(f - p \pdv{f}{p}) + b\qty(g - p \pdv{g}{p}) \\
        &= aE[f] + bE[g].
    \end{align*}
    For (2), if $f$ is homogeneous of degree $k$ in $p$, then by Euler's homogeneous function theorem \cite{apostol1973mathematical}, we have
    \begin{equation*}
        p \pdv{f}{p} = k f.
    \end{equation*}
    Substituting this into the definition of $E$, we get
    \begin{equation*}
        E[f] = f - p \pdv{f}{p} = f - k f = (1 - k)f.
    \end{equation*}
    For (3), if $f$ is polynomial in $p$, it can be expressed as
    \begin{equation*}
        f = \sum_{k=0}^{d} f_k
    \end{equation*}
    where each $f_k$ is homogeneous of degree $k$ in $p$. Applying $E$ to $f$, we have
    \begin{align*}
        E[f] &= E\qty[\sum_{k=0}^{d} f_k] = \sum_{k=0}^{d} E[f_k] \\
        &= \sum_{k=0}^{d} (1 - k) f_k
    \end{align*}
    where we used property (2) for each homogeneous component $f_k$. This completes the proof.
\end{proof}
\end{prop}

\textit{Remark:} An important consequence of Proposition \ref{prop:euler_op_properties} is that the Euler operator $E$, in general, preserves the degree of polynomials in the momentum variables $p$. Specifically, if $f$ is a monomial of degree $d$ in $p$, then $E[f]$ is also a monomial of degree $d$ in $p$, unless $f$ is homogeneous of degree one, in which case $E[f] = 0$.

\subsection{Bracket Technicalities}
\label{app:Bracket_Technicalities}

Here we collect some technical results regarding the contact-Jacobi bracket that are used in the main text.

\begin{lemma}\label{lem:bracket_operators}
    Given a monomial $f(x,p,u) = \gamma p^\alpha\in\mfsc(J^1(\R^n))$ in the momentum variables $p$ of degree $|\alpha| = k\geq 2$ with constant coefficient $\gamma\in\R$, each of the following operations can be written in terms of contact-Jacobi brackets with elements from $\mfpd$ and $\mfsc$:
    \begin{enumerate}
        \item Degree raising
        \item Degree lowering
        \item Scalar multiplication by a smooth function of $h(x,u)$, modulo lower degree terms.
    \end{enumerate}
\end{lemma}

\begin{proof}
    Recall that the contact-Jacobi bracket of two functions $g,f \in C^{\infty}(J^1(\R^n))$ is given by
    \begin{equation*}
        \qty[g,f] = \qty{g,f} + \pdv{g}{u}E[f] - \pdv{f}{u}E[g],
    \end{equation*}

    1. Degree raising: Let $f$ be the monomial as above, and set $g = u(1-\abs{\alpha})^{-1}p_i$. Then, we have
    \begin{align*}
        \qty[g,f] &= \qty[u(1-\abs{\alpha})^{-1}p_i, \gamma p^\alpha]\\
        &= \gamma (1-\abs{\alpha})^{-1} \qty{up_i, p^\alpha} + \gamma (1-\abs{\alpha})^{-1} \pdv{(up_i)}{u}E[p^\alpha] - 0 \\
        &= \gamma (1-\abs{\alpha})^{-1} \qty{up_i, p^\alpha} + \gamma (1-\abs{\alpha})^{-1} p_i (1 - \abs{\alpha}) p^\alpha \\
        &= \gamma (1-\abs{\alpha})^{-1} \qty{up_i, p^\alpha} + \gamma p_i p^\alpha
    \end{align*}
    Now, we compute the Poisson bracket term:
    \begin{align*}
        \qty{up_i, p^\alpha} &= \sum_{j}\qty(\pdv{(up_i)}{p_j} \pdv{(p^\alpha)}{x_j} - \pdv{(up_i)}{x_j} \pdv{(p^\alpha)}{p_j}) \\
        &= u \pdv{(p^\alpha)}{x_i} - 0 \\
        &= u \cdot 0 = 0
    \end{align*}
    since $p^\alpha$ does not depend on $x$. Thus, we have
    \begin{equation*}
        \qty[g,f] = \gamma p_i p^\alpha = \gamma p^{\alpha + e_i}
    \end{equation*} 
    
    2. Degree lowering: Let $f$ be the monomial as above, and set $g=-x_i(\alpha_i)^{-1}$, where $\alpha_i$ is the $i$-th component of the multi-index $\alpha$. Then, we have
    \begin{align*}
        \qty[g,f] &= \qty{-x_i(\alpha_i)^{-1}, \gamma p^\alpha}\\
        &= \gamma (\alpha_i)^{-1} \qty{-x_i, p^\alpha} + 0 - 0 \\
        &= \gamma (\alpha_i)^{-1} \pdv{p^\alpha}{p_i} \\
        &= \gamma p^{\alpha - e_i}
    \end{align*}
     
    3. Scalar multiplication: Let $f$ be the monomial as above, with $\abs{\alpha}=k\geq 2$, and let $h(x,u) \in C^{\infty}(\R^{n+1})$ be a smooth function of $(x,u)$. We can express the scalar multiplication $h(x,u) f$ using the contact-Jacobi bracket as follows. Let $g$ be an appropriately scaled antiderivative of $h$ with respect to $u$, i.e.:
    \begin{equation*}
        g(x,u)=\frac{1}{1-k}\int_{u_0}^{u} h(x,s) ds
    \end{equation*}
    Then, we have
    \begin{align*}
        \qty[g,f] &= \qty[g, \gamma p^\alpha]\\
        &= \gamma \qty{g, p^\alpha} + \gamma \pdv{g}{u}E[p^\alpha] - 0 \\
        &= \gamma \qty{g, p^\alpha} + \gamma \frac{1}{1-k} h(x,u) (1-k) p^\alpha\\
        &= \gamma \qty{g, p^\alpha} + \gamma h(x,u) p^\alpha
    \end{align*}
    Now, we compute the Poisson bracket term:
    \begin{align*}
        \qty{g, p^\alpha} &= 0 - \pdv{g}{x_j} \pdv{(p^\alpha)}{p_j} \\
        &= -\pdv{g}{x_j} \alpha_j p^{\alpha - e_j}
    \end{align*}
    Thus, we have
    \begin{equation*}
        \qty[g,f] = \gamma h(x,u) p^\alpha - \gamma \sum_{j=1}^{n} \pdv{g}{x_j} \alpha_j p^{\alpha - e_j}
    \end{equation*}
    which expresses the scalar multiplication $h(x,u) f$ modulo lower degree terms.
\end{proof}

An immediate consequence of \cref{lem:bracket_operators} is that one does not need the full algebra of strict contactomorphisms to generate polynomials in the momentum variables via repeated brackets with prolonged diffeomorphisms. Just the second degree monomials in $p$ are sufficient.

We finally note the following lemma regarding the degree-preserving nature of the full contact-Jacobi bracket.

\begin{lemma}\label{lem:bracket_degree}
     Given two polynomials $f,g \in C^{\infty}(J^1(\R^n))$ in the momentum variables $p$ of degree at most $k$ and $m$ respectively, their contact-Jacobi bracket $\qty[f,g]$ is a polynomial in $p$, of  degree at most $k+m$.
    \begin{proof}

    The contact-Jacobi bracket of $f$ and $g$ is given by
    \begin{equation*}
        \qty[f,g] = \qty{f,g} + \pdv{f}{u}E[g] - \pdv{g}{u}E[f],
    \end{equation*}
    where $\{f,g\}$ is the Poisson bracket and $E$ is the Euler operator.

    First, let us treat a monomial case for clarity. Let $f,g$ be monomials in $p$ of degree $m$ and $k$ respectively:
    \begin{equation*}
        f(x,p,u) = a(x,u) p^{\alpha}, \quad g(x,p,u) = b(x,u) p^{\beta}
    \end{equation*}

    The Poisson bracket term will either vanish or contribute a term of degree $k+m-1$ in $p$. The Euler operator, by \cref{prop:euler_op_properties} , preserves the degree of polynomials in $p$. Therefore, the terms $\pdv{f}{u}E[g]$ and $\pdv{g}{u}E[f]$ will also contribute terms of degree $k+m$ in $p$. Overall, the contact-Jacobi bracket $\qty[f,g]$ will be a polynomial in $p$ of degree at most $k+m$ under generic conditions.

    Now, let $f,g \in C^{\infty}(J^1(\R^n))$ be polynomials in the momentum variables $p$ of degree at most $m$ and $k$ respectively. We can express them as
    \begin{equation*}
        f = \sum_{\abs{\alpha} \leq m} f_{\alpha}(x,u)p^{\alpha}, \quad g = \sum_{\abs{\beta} \leq k} g_{\beta}(x,u)p^{\beta}
    \end{equation*}
    where $f_{\alpha}(x,u)$ and $g_{\beta}(x,u)$ are smooth functions of $(x,u)$. By our previous argument, the contact-Jacobi bracket of each pair of monomials $f_{\alpha}(x,u)p^{\alpha}$ and $g_{\beta}(x,u)p^{\beta}$ will be a polynomial in $p$ of degree at most $\abs{\alpha} + \abs{\beta}$. Since $\abs{\alpha} \leq m$ and $\abs{\beta} \leq k$, it follows that $\abs{\alpha} + \abs{\beta} \leq m + k$. Therefore, the contact-Jacobi bracket $\qty[f,g]$ will be a polynomial in $p$ of degree at most $m + k$, unless it vanishes.
\end{proof}
\end{lemma}

\subsection{Proof of Proposition 3.1}

Since polynomials in the momentum variables $p$ are dense in the space of smooth functions on $J^1(\R^n)$ in the $C^r$ norm by \cref{thm:density}, it suffices to show that any polynomial-in-$p$ Hamiltonian belongs to $\mfg$. By linearity of $\mfg$, it then suffices to show that any monomial of the form $\gamma f(x,u) p^\alpha$ with $\gamma\in\R$ and $\alpha$ a multi-index of finite degree, can be generated by repeated brackets of strict contact Hamiltonians and prolonged Hamiltonians. 

If $\abs{\alpha}\leq 1$, the monomial is already affine in $p$, and therefore belongs to the Lie algebra generated by prolonged Hamiltonians. We then need to address the case of $\abs{\alpha} \geq 2$. Monomials of the form $\gamma p^\alpha$ for $\gamma\in\R$ and $\alpha$ a multi-index of finite degree already belong to the Lie algebra generated by strict contact Hamiltonians. By \cref{lem:bracket_operators}, there exists an element of $d\in\mfpd$ such that
\begin{align*}
    \qty[d, \gamma p^\alpha] = \gamma f(x,u) p^\alpha + \text{lower degree terms}
\end{align*}
where $f(x,u)$ is an arbitrary smooth function of $(x,u)$.

In order to isolate the leading term $\gamma f(x,u) p^\alpha$ by removing the lower degree terms (by linearity), we iteratively apply the same generating procedure to the highest degree monomial among the lower degree terms. This yields a triangular system of equations, which can be solved recursively. Suppose that the highest degree of the lower degree terms is $k < \abs{\alpha}$. For each monomial of the form $\gamma' g(x,u) p^\beta$ with $\abs{\beta} = k$, we can find an element $d'\in\mfpd$ such that \begin{align*}
    \qty[d', \gamma' p^\beta] = \gamma' g(x,u) p^\beta + \text{lower degree terms}
\end{align*}
and we can remove the monomial $\gamma' g(x,u) p^\beta$ from the lower degree terms by linearity. We can repeat this procedure until the base case of $\abs{\beta}=1$ which is already handled, and we are left with only the leading term $\gamma f(x,u) p^\alpha$. This shows that any monomial of the form $\gamma f(x,u) p^\alpha$ can be generated by repeated brackets of strict contact Hamiltonians and prolonged Hamiltonians, which completes the proof. 

\textbf{Remark:} We have only used the scalar-multiplication property of \cref{lem:bracket_operators}. This allows us to solve the triangular system of equations that arises when trying to isolate the leading term \text{with only depth-1 brackets!}. Alternatively, one could also use the degree-raising and degree-lowering properties to solve the triangular system by induction on the degree of the monomials, which would require deeper brackets.

Note that, as defined, our two subalgebras $\mfsc$ and $\mfpd$ are enormous, and they intersect nontrivially. For example, any monomial of the form $f(x)p$ belongs to both subalgebras. If we made use of bracket raising and lowering operations, we could generate the entire algebra of polynomials in $p$ by repeated brackets of just the second degree monomials in $p$ with prolonged Hamiltonians, which is a much smaller set of generators. The subalgebras could possibly be coarsened further to a smaller set of generators, which may be of value from a practical perspective. Note, for example, that the generators from $\mfpd$ and $\mathrm{span}\qty{p^\alpha: \abs{\alpha}=2}$ form subalgebras that do not intersect, and would therefore generate a graded Lie algebra of polynomials in $p$ by repeated brackets.

\section{Fundamentals of Geometric Integration}
\label{ap:Geometric_Integration}

\subsection{Gr\"onwall Estimates}
\label{ap:Gronwall}

Our numerical approach is based on approximating an $r+1$-smooth Hamiltonian $H\in C^{r+1}(U)$, over a compact set $U\subset J^1(\R^n)$, by a polynomial surrogate $H^{(N)}$ in the momentum variables $p$ satisfying
\begin{align*}
    \norm{H - H^{(N)}}_{C^{r+1}(U)} < \varepsilon_N
\end{align*}
where the $C^{r+1}$ norm is defined as
\begin{equation*}
    \norm{H - H^{(N)}}_{C^{r+1}(U)} = \max_{0\leq k \leq r+1} \sup_{z\in U} \norm{D^k H(z) - D^k H^{(N)}(z)}
\end{equation*}

The approximation $H^{(N)}$ exists due to \cref{thm:density}, and can be obtained constructively by, for example, the Bernstein or Chebyshev polynomial approximations \cite{devore1993constructive}; polynomials in each $(x,u,p)$-coordinate are trivially polynomials in $p$. A Taylor approximation could also be valid locally, but would not guarantee a uniform approximation over the entire compact set $U$.

The main source of error in our numerical integrators is the error between the true flow of $H$ and the flow of $H^{(N)}$, which is a contactomorphism that can be approximated by our splitting algebra. This is governed by the following Gr\"onwall-type estimate which is standard in the dynamical systems literature, and can be derived by applying the standard Gr\"onwall inequality to the difference of the two flows.

\begin{prop}\label{prop:gronwall}
    Let $H\in C^{r+1}(U)$ be an $r+1$-smooth Hamiltonian defined on a compact set $U\subset J^1(\R^n)$, and let $H^{(N)}$ be a polynomial in the momentum variables that approximates $H$ in the $C^{r+1}$ norm as above. Assume that for $0\leq t\leq T$ the flows of $H$ and $H^{(N)}$ exist and remain in $U$. Let $L$ be a Lipschitz constant for the vector fields generated by $H$ and $H^{(N)}$ on $U$, and assume that each derivative of $H$ and $H^{(N)}$ up to order $r+1$ is bounded on $U$:
    \begin{equation*}
        \norm{D^k H}_{C^0(U)} \leq M_k, \quad \norm{D^k H^{(N)}}_{C^0(U)} \leq M_k^{(N)} \quad \text{for } 0\leq k \leq r+1
    \end{equation*} 
    Then, the flows $\Phi_H^t$ and $\Phi_{H^{(N)}}^t$ satisfy the following estimate for all $0 \leq t \leq T$:
    \begin{equation}
        \norm{\Phi_H^t - \Phi_{H^{(N)}}^t}_{C^r(U)} \leq C_r P_r(t) e^{Lt}\varepsilon_N
    \end{equation}
    where $C_r$ is a constant depending on $r$ and the derivative bounds $M_k$ and $M_k^{(N)}$, and $P_r(t)$ is a polynomial in $t$ of degree at most $r$, whose coefficients depend on the same bounds together with the Lipschitz constant $L$.
    \begin{proof}[Sketch]
        The proof relies on the standard Gr\"onwall inequality applied sequentially to the derivatives of the difference of the flows $\Phi_H^t$ and $\Phi_{H^{(N)}}^t$, where the contact vector fields generated by the two Hamiltonians satisfy:
        \begin{align*}
            \norm{X_H-X_{H^{(N)}}}_{C^r(U)}\leq C \norm{H-H^{(N)}}_{C^{r+1}(U)} \leq C \varepsilon_N
        \end{align*}
        The result is a triangular system of inequalities that bounds the difference of the flows in terms of the initial error $\varepsilon_N$ and the Lipschitz constant $L$, as well as the $C^r$ norms of $H$ and $H^{(N)}$, which contribute to the constant $C_r$ and the polynomial $P_r(t)$. See \cite[Chapter 5]{hartman2002ordinary} or \cite[Chapter 4]{abraham2012manifolds} for a more detailed argument.
    \end{proof}
\end{prop}

\subsection{Backward Error Analysis}

The Baker-Campbell-Hausdorff (BCH) formula is a fundamental tool in the analysis of Lie group flows and their numerical approximations. It provides a way to express the composition of exponentials of Lie algebra elements in terms of a single exponential, which is crucial for understanding the behavior of numerical integrators that are based on Lie group methods.

\begin{lemma}[BCH]
    \label{lem:bch}
    Let $\mathfrak{g}$ be a Lie algebra and let $X,Y \in \mathfrak{g}$. Then, there exists a formal series expansion for the element $Z \in \mathfrak{g}$ such that
    \begin{equation}
        e^X e^Y = e^Z
    \end{equation}
    where $Z$ can be expressed as a series in $X$ and $Y$ involving nested commutators:
    \begin{equation}
        Z = X + Y + \frac{1}{2}[X,Y] + \frac{1}{12}[X,[X,Y]] - \frac{1}{12}[Y,[X,Y]] + \cdots
    \end{equation}
    where the series continues with higher-order nested commutators of $X$ and $Y$.
\end{lemma}

This result is powerful because it is entirely Lie-theoretic and therefore applies to arbitrary Lie algebras. In particular, many results from the geometric integration of symplectic systems that rely on BCH transfer directly to our setting, with the contact-Jacobi bracket replacing the Poisson bracket.

In particular, we obtain the following formal statement:

\begin{corollary}[Contact BCH]
    \label{cor:contact_BCH}
    Let $H_1,...,H_m \in C^{\infty}(J^1(\R^n))$ be a collection of contact Hamiltonians, and let $\exp(hH_i)$ denote the time-$h$ map generated by $H_i$. Then, there exists a formal contact Hamiltonian series $\tilde{H}(h) \in C^{\infty}(J^1(\R^n))[[h]]$ such that
    \begin{equation}
        \exp(hH_m) \circ... \circ \exp(hH_2) \ \circ \exp(hH_1) = \exp(h\tilde{H}(h))
    \end{equation}
    where $\tilde{H}(h)$ is a formal series in $h$, whose coefficients are determined by the $H_i$ and their nested contact-Jacobi brackets, analogous to the BCH expansion.
\end{corollary}

The Hamiltonian $\tilde{H}$ is often referred to as the \emph{modified Hamiltonian} or \emph{effective Hamiltonian} of the composition of flows. It captures the cumulative effect of the individual flows and their interactions, and is crucial for understanding the long-term behavior of numerical integrators based on compositions of flows. Without additional regularity assumptions, the modified Hamiltonian $\tilde{H}$ is only guaranteed to be a formal power series in the $H_i$ and their brackets, and may not converge.

Under analyticity assumptions, one has the usual backward-error/shadowing conclusion; we record a contact-formulation here for completeness:

\begin{theorem}[Analytic BEA] 
    \label{thm:analytic_BEA}
    Let $H$ be a real analytic Hamiltonian defined on a compact subset $U\subset J^1(\R^n)$, which extends holomorphically to a complex neighborhood of $U$ with radius of convergence $\rho > 0$. Denote its flow by $\Phi_H^t$, and assume it remains in $U$ for all $t \leq T$. Let $\Psi_h$ be an analytic contact one-step numerical integrator that approximates the flow of $H$ with local error of order $p$, i.e., $\Psi_h = \Phi_H^h + O(h^{p+1})$. Then, there exists a modified Hamiltonian $\tilde{H}_h$, which is also real analytic on $U'\subseteq U$ and extends holomorphically to a complex neighborhood of $U'$ with radius of convergence $0<\rho'\leq \rho$, such that $$\tilde{H}_h = H + O(h^p).$$ Furthermore, assume that the iterates $\Psi_h^n$ and $\Phi_{\tilde{H}_h}^{nh}$ remain in $U'$ for all $nh\leq T$. Then,
    there exist constants $C,c$ and a positive integer $N$ satisfying $N \leq T/h$ such that for all $n \leq N$ we have
    \begin{equation}
        \norm{\Psi^{n}_h - \Phi^{nh}_{\tilde{H}_h}}_{\rho'} \leq Cn e^{-c/h}
    \end{equation}
\end{theorem}
Note that the norm here is the analytic supremum norm associated with the radius $\rho'$. The constants $C$ and $c$ only depend on the analyticity properties of $H$ and the local error of the integrator, and are independent of the step size $h$. Since $N = O(T/h)$, the estimate controls the error on any fixed finite time interval by taking $h$ sufficiently small. Under further assumptions on the confinement of the flow and the integrator iterates within an invariant compact analytic domain $K\subset U$, one may expect to extend the error estimate to all $nh \leq e^{\kappa/h}$, where $\kappa$ is a positive constant. Such results are standard in the symplectic setting, and would imply that the integrator faithfully reproduces the flow of the modified Hamiltonian $\tilde{H}_h$ over exponentially long times.

For symplectic integrators, analytic BEA is often interpreted through near-conservation of a modified energy. In the contact setting, the analogous quantity is the conformal factor rather than an exactly conserved energy. Since this consequence is useful later, we record the corresponding conformal-factor estimate below.

\begin{corollary}[Modified Conformal Factor]\label{thm:conformal_tracking}
    Suppose that $H$ and $\tilde{H}_h$ are the original and modified contact Hamiltonians appearing in \cref{thm:analytic_BEA}, and that their flows remain in a compact set $U\subset J^1(\R^n)$ for all $t \leq T$. Then, the conformal factors $\lambda(t)$ and $\tilde{\lambda}(t)$ of the flows generated by $H$ and $\tilde{H}_h$, respectively, satisfy
    \begin{equation}
        \tilde{\lambda}(t) = \lambda(t) + O(h^p t)
    \end{equation}
    for all $t \leq T$.
    \begin{proof}
        Under the assumptions of the Analytic BEA theorem, we have $\tilde{H}_h = H + O(h^p)$ in an analytic norm. By a Cauchy estimate, it follows that the difference $\tilde{H}_h - H$ and all of its derivatives are also bounded by $O(h^p)$ in a $C^r$ norm for any finite $r$, on a slightly smaller compact set $K\subset U$. Taking the partial derivative with respect to $u$, we therefore obtain
    \begin{equation}
        \pdv{\tilde{H}_h}{u} = \pdv{H}{u} + O(h^p)
    \end{equation}
    Now, let $\Phi_H^t$ and $\Phi_{\tilde{H}_h}^t$ be the flows generated by $H$ and $\tilde{H}_h$, with conformal factors $e^{\lambda(t)}$ and $e^{\tilde{\lambda}(t)}$, respectively. The logarithmic conformal factors are given by
    \begin{equation*}
        \lambda(t)(z) = -\int_0^t \partial_u H\qty(\Phi_H^s(z)) ds\qc
        \tilde{\lambda}(t)(z)= -\int_0^t \partial_u \tilde{H}_h\qty(\Phi_{\tilde{H}_h}^s(z)) ds
    \end{equation*}
    with $z\in U$. Using the fact that $\partial_u{\tilde{H}_h} = \partial_u{H} + O(h^p)$, we can express $\tilde{\lambda}(t)$ in terms of $\lambda(t)$ as follows:
    \begin{align*}
        \tilde{\lambda}(t)(z) &= -\int_0^t \partial_u{\tilde{H}_h}\qty(\Phi_{\tilde{H}_h}^s(z)) ds \\
        &= -\int_0^t \qty( \partial_u{H}(\Phi_{\tilde{H}_h}^s(z)) + O(h^p) ) ds \\
        &= -\int_0^t \partial_u{H}(\Phi_{\tilde{H}_h}^s(z)) ds + O(h^p t)
    \end{align*}
    Now, since $\Phi_{\tilde{H}_h}^s$ approximates $\Phi_H^s$ with an error of order $O(h^p)$ for all $s \leq t$ (the same Gr\"onwall-type argument as in \cref{prop:gronwall} applies to $H$ and $\tilde{H}_h$), we can express $\Phi_{\tilde{H}_h}^s$ in terms of $\Phi_H^s$ as follows:
    \begin{equation*}
        \Phi_{\tilde{H}_h}^s(z) = \Phi_H^s(z) + O(h^p s),
    \end{equation*}
    and thus, substituting this into the expression for $\tilde{\lambda}(t)$, we get
    \begin{align*}
        \tilde{\lambda}(t)(z) &= -\int_0^t \partial_u{H}(\Phi_H^s(z) + O(h^p s)) ds + O(h^p t) \\
        &= -\int_0^t \qty( \partial_u{H}(\Phi_H^s(z)) + O(h^p s) ) ds + O(h^p t) \\
        &= -\int_0^t \partial_u{H}(\Phi_H^s(z)) ds + O(h^p t) \\
        &= \lambda(t)(z) + O(h^p t)
    \end{align*}
    \end{proof}
\end{corollary}

The previous corollary is again a Gr\"onwall-type estimate, giving a bound on the difference between the conformal factors of two flows generated by nearby Hamiltonians. The discrete analogue then follows from the same shadowing argument, so we record it here for completeness.

\begin{corollary}[Discrete Conformal Tracking]\label{cor:discrete_conformal_tracking}
     Let $\Psi_h$ be the time-$h$ map of an analytic contact one-step numerical integrator that approximates the flow of a real analytic Hamiltonian $H$ with local error of order $p$. Assume that $H$, $\tilde{H}_h$, and $\Psi_h$ satisfy the assumptions of \cref{thm:analytic_BEA}, and that the exact flow $\Phi_H^t$, the numerical iterates $\Psi_h^k$, and the shadow iterates $\Phi_{\tilde{H}_h}^{kh}$ remain in a compact set $U\subset J^1(\R^n)$ for all $t \leq T$ and $kh \leq T$. Then, the cumulative discrete conformal factor $\sigma_h^n$ of $\Psi_h$ after $n$ steps satisfies
    \begin{align*}
        \sigma_h^n(z) &= \tilde{\lambda}(nh)(z)+O((T/h)e^{-c/h})
    \end{align*}
    for all $n \leq N$, where $\tilde{\lambda}(t)$ is the conformal factor of the flow generated by the modified Hamiltonian $\tilde{H}_h$ and $c$ is a positive constant independent of $h$. In particular, since $\tilde{\lambda}(t) = \lambda(t) + O(h^p t)$, we have
    \begin{equation}
        \sigma_h^n(z) = \lambda(nh)(z) + O(h^p T) + O((T/h)e^{-c/h})
    \end{equation}
     where $\lambda(t)$ is the conformal factor of the flow generated by $H$.
\end{corollary}

\begin{proof}
    We consider the discrete time-$h$ map $\Psi_h$ generated by the integrator, which has an associated discrete conformal factor $e^{\sigma_h}(z)$. The cumulative discrete conformal factor after $n$ steps is given by
    \begin{equation}
        \sigma_h^n(z) = \sum_{k=0}^{n-1} \sigma_h(\Psi_h^k(z))
    \end{equation}
    Since the Analytic BEA theorem gives an exponential shadowing estimate for $\Psi_h$ and $\Phi_{\tilde{H}_h}^t$ in an analytic norm, a Cauchy bound again converts this into a $C^r$ estimate for any finite $r$ on a slightly smaller compact set $K\subset U$. Thus, we can express $\Psi_h^k$ in terms of $\Phi_{\tilde{H}_h}^{kh}$ as follows:
    \begin{equation*}
        \norm{\Psi_h^k - \Phi_{\tilde{H}_h}^{kh}}_{C^r(K)} \leq C ke^{-c/h}.
    \end{equation*}
    Consequently,
    \begin{align*}
        \norm{\Psi_h^*\alpha - (\Phi_{\tilde{H}_h}^h)^*\alpha}_{C^0(K)} &\leq C' e^{-c/h}
    \end{align*}
    with $C'$ depending on the $C^1$ norm of $\alpha$ over $K$ along with the previous bound. Letting $R=\pdv{u}$ be the Reeb vector field on $J^1(\R^n)$, we directly evaluate:
    \begin{align*}
        \abs{e^{\sigma_h}-e^{\tilde{\lambda}(h)}}&=\abs{(e^{\sigma_h} - e^{\tilde{\lambda}(h)})\alpha R}\\&=\abs{(\Psi_h^*\alpha) R - \qty((\Phi_{\tilde{H}_h}^h)^*\alpha) R}  \\
        &\leq C' e^{-c/h}
    \end{align*}
    It follows that the one-step conformal factor $\sigma_h$ satisfies
    \begin{align*}
        \sigma_h(z) &= \tilde{\lambda}(h)(z) + O(e^{-c/h})\\
        &=-\int_0^h \partial_u{\tilde{H}_h}\qty(\Phi_{\tilde{H}_h}^s(z)) ds + O(e^{-c/h}).
    \end{align*}
    Applying the same argument to the iterates of $\Psi_h$, we obtain:
    \begin{align*}
        \tilde{\lambda}(h)(\Psi_h^k(z)) &= -\int_0^h \partial_u{\tilde{H}_h}\qty(\Phi_{\tilde{H}_h}^s(\Psi_h^k(z))) ds\\
        &=-\int_0^h \partial_u{\tilde{H}_h}\qty(\Phi_{\tilde{H}_h}^s(\Phi_{\tilde{H}_h}^{kh}(z))) ds + O(e^{-c/h}) \\
        &=-\int_0^h \partial_u{\tilde{H}_h}\qty(\Phi_{\tilde{H}_h}^{s+kh}(z)) ds + O(e^{-c/h})\\
        &= \tilde{\lambda}(h)(\Phi_{\tilde{H}_h}^{kh}(z)) + O(e^{-c/h})
    \end{align*}
    We therefore conclude that the cumulative discrete conformal factor $\sigma_h^n$ satisfies
    \begin{align*}
    \sigma_h^n(z) &= \sum_{k=0}^{n-1} \sigma_h(\Psi_h^k(z)) \\
    &=\sum_{k=0}^{n-1}\tilde{\lambda}(h)(\Phi_{\tilde{H}_h}^{kh}(z)) + O(n e^{-c/h}) \\
    &=-\sum_{k=0}^{n-1} \int_0^h \partial_u{\tilde{H}_h}\qty(\Phi_{\tilde{H}_h}^{s+kh}(z)) ds + O((T/h)e^{-c/h}) \\
    &=-\int_0^{nh} \partial_u{\tilde{H}_h}\qty(\Phi_{\tilde{H}_h}^{s}(z)) ds + O((T/h)e^{-c/h}) \\
    &= \tilde{\lambda}(nh)(z) + O((T/h)e^{-c/h})
    \end{align*}

    \end{proof}

\section{Extended Numerical Results}
\label{ap:Extended_Numerical}

\subsection{Damped Harmonic Oscillator}
\label{ap:DHO_extended}

While our theoretical results apply to an arbitrary choice of Hamiltonian, the representation one obtains via the generators of the splitting algebra is not unique. Thus, while the \textit{order} of the integrator may be common, there is still a significant design choice to be made in the choice of splitting, which will affect the numerical performance of the scheme. In particular, the constants associated with the error bounds we have discussed in our main approximation theorems depend on the commutators of the generators, and therefore on the choice of splitting.

In \cref{fig:DHO_convergence_extended} we compare the full-state trajectory error for the damped harmonic oscillator for several different choices of splitting, for a fixed initial condition. Overall, we observe the symmetric Strang splitting SPS to perform better than all lower order Lie-Trotter splittings. However, there are noticeable differences between splittings of the same time as well - based on the choice of generators. In particular, we notice that Splitting 1 (using only the $\frac{1}{2}p^2$ as the symplectic generator ) performs better than Splitting 2 (using $\frac{1}{2}(p^2+x^2))$ as the "symplectic part". This performance gap can depend on the initial condition, and the part of the domain explored by a trajectory (i.e. the pointwise evaluation of the brackets), but it is a clear demonstration of the fact that the choice of splitting can have a significant impact on the numerical performance of the integrator, even under our theoretical framework. This is a well-known phenomenon in the geometric integration literature, and is an important consideration when designing integrators for specific applications.

\begin{figure}
    \centering
    \includegraphics[width=0.9\textwidth]{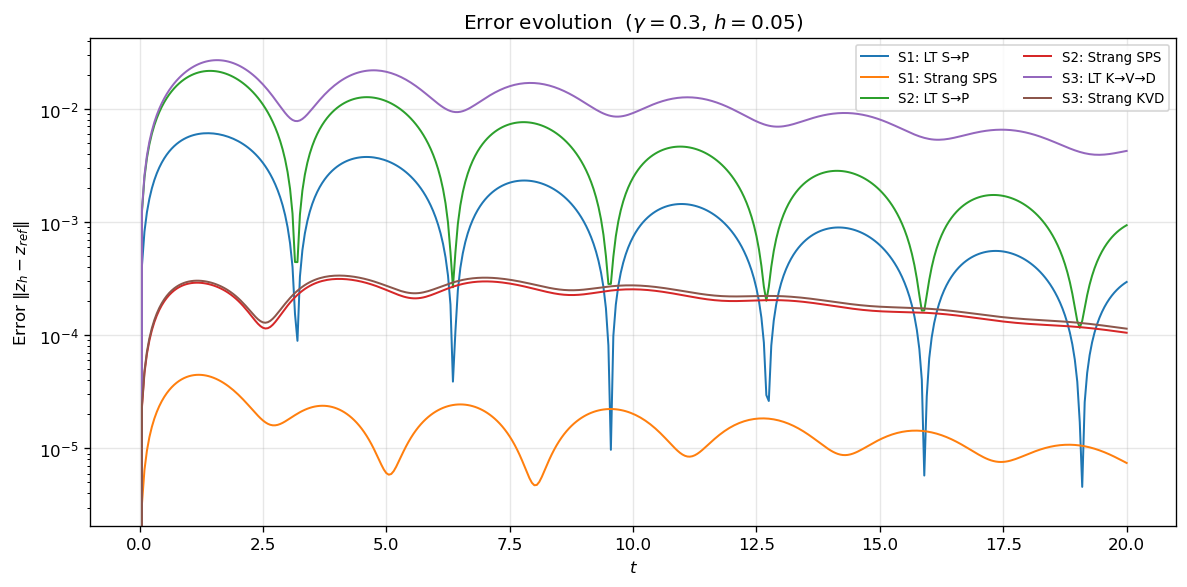}
    \caption{Extended convergence results for the damped harmonic oscillator, including a comparison of different choices of splitting. From left to right: exact reference (Bernoulli solve), splitting $A$ (standard), splitting $B$ (alternative).}
    \label{fig:DHO_convergence_extended}
\end{figure}

\subsection{VdP Oscillator}
\label{ap:VdP_extended}

In the main text, we only include the attractor projected on the $(x,u)$-plane. We further visualize the $(x,p)$-plane projection of the attractor for the forced VdP system, where the $p$-coordinate is obtained using automatic differentiation. We visualize the attractor for different choices of projection, integrator step-size $h$, and forcing parameter $\omega$ in \cref{fig:VdP_attractors_extended}. The latter controls whether the behavior is chaotic or periodic, and we can see that the attractor is more complex in the chaotic regime. 

\begin{figure}
    \centering
    \includegraphics[width=0.9\textwidth]{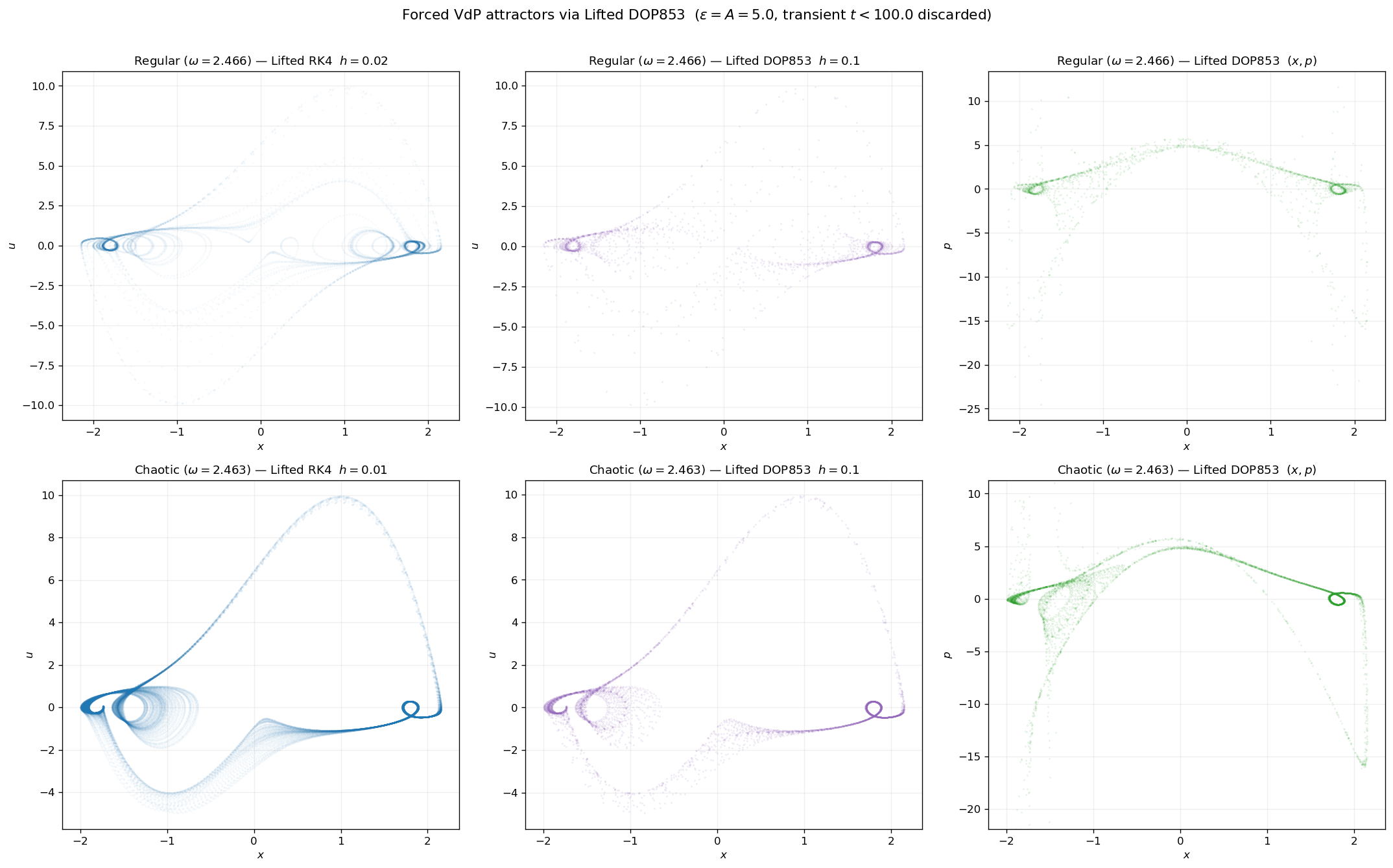}
    \caption{Visualization of the attractor for the forced VdP system for different choices of projection ($(x,u)$ vs. $(x,p)$), integrator step-size $h$ and forcing parameter $\omega$.}
    \label{fig:VdP_attractors_extended}
\end{figure}

\subsection{Nonlinear System}
\label{ap:Nonlinear_System_extended}

In the main text, we only present results for the simplest commutator gadget. We consider two additional commutator gadgets to approximate the flow of the nonlinear $p^2u$ damping term, a symmetric gadget (Splitting $D$) and a Yoshida-type gadget (Splitting $E$), which are defined in \cref{ap:Numerical_Integrators}. At the cost of additional computational overhead, these gadgets achieve first-order global accuracy. In \cref{fig:Gadgets_convergence} we showcase the convergence behavior for the three different gadgets, and in \cref{fig:Gadgets_attractors} we visualize the attracting set for each choice of gadget compared to the one obtained by an analytic splitting of the contact Hamiltonian. Overall, the qualitative behavior of the attractor is similar for all three gadgets, but the symmetric and Yoshida-type gadgets achieve a better approximation (and asymptotic order).

For the construction of higher-order integrators (including for approximating flows of commutators), see \cite{yoshida1990construction,suzuki1990fractal} and for a contact specific version, see \cite[Propositions 3.2, 3.3]{zadra2021vdp}

\begin{figure}
    \centering
    \includegraphics[width=0.9\textwidth]{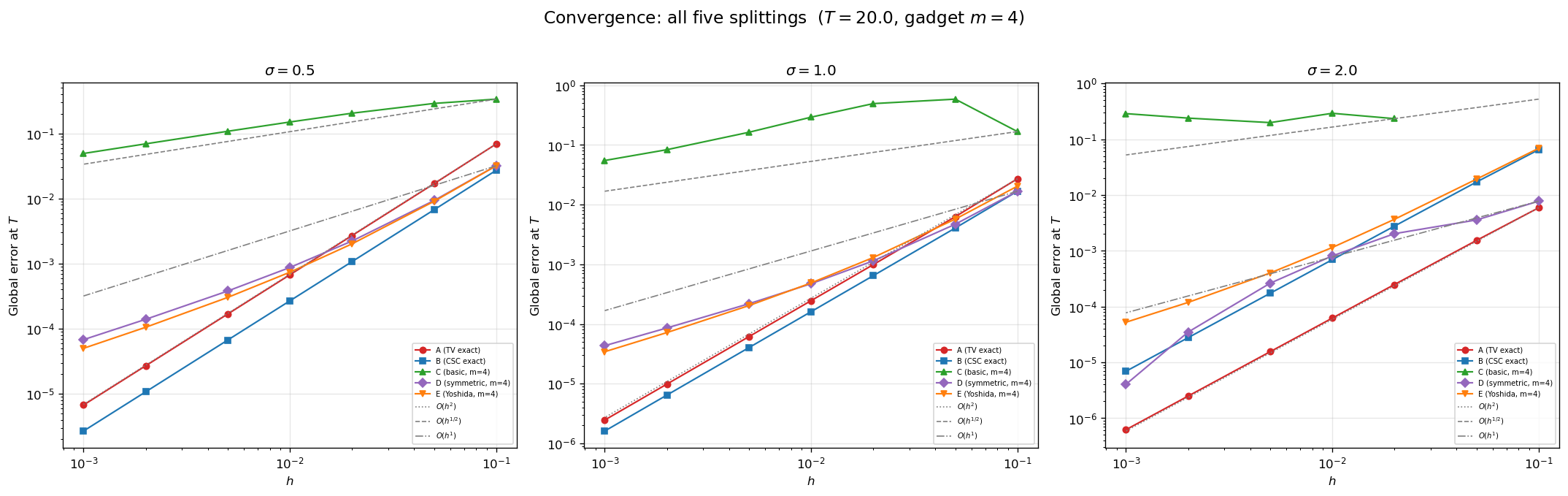}
    \caption{Extended convergence results for the nonlinear system, including the symmetric and Yoshida-type commutator gadgets, for different choices of damping parameter value $\sigma$.}
    \label{fig:Gadgets_convergence}
\end{figure}

\begin{figure}
    \centering
    \includegraphics[width=0.9\textwidth]{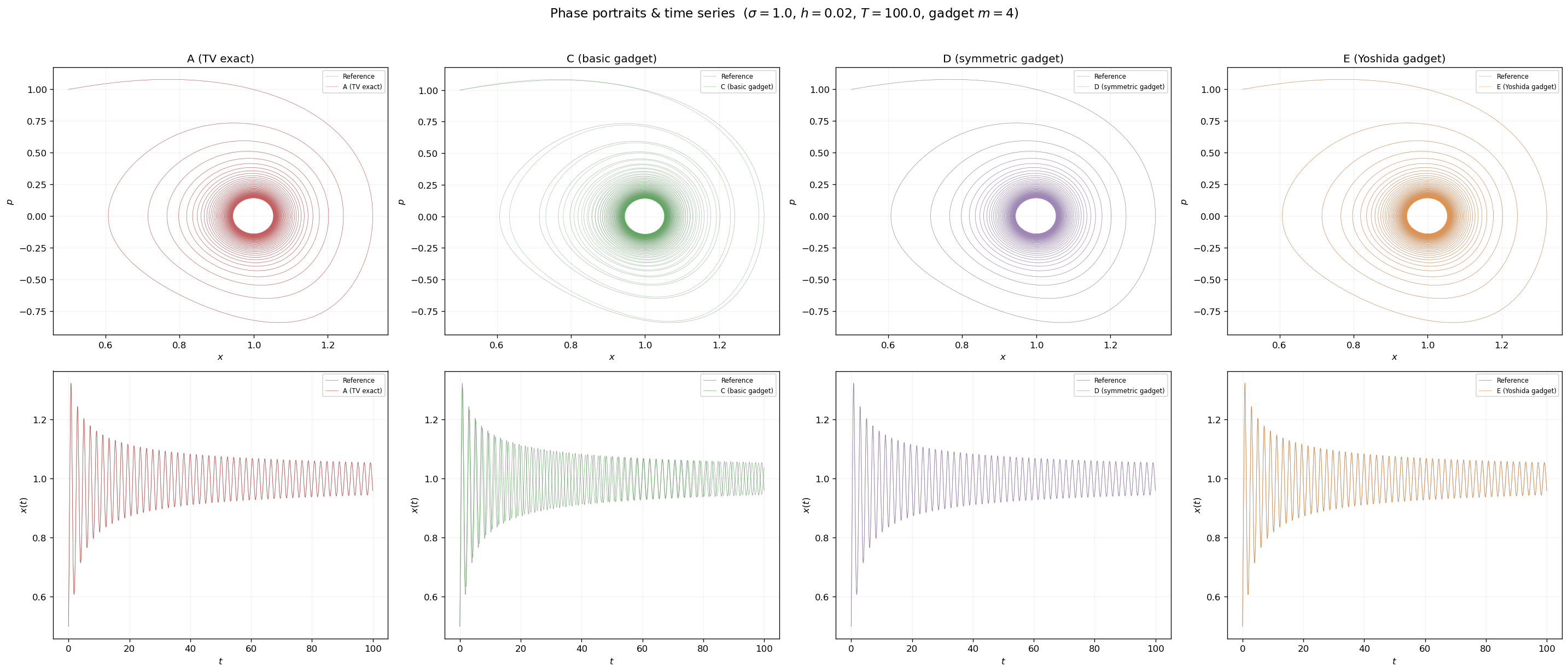}
    \caption{Visualization of the attracting set for different choices of the commutator gadget term. From left to right: exact reference (Bernoulli solve), standard commutator gadget $(O(h^{1/2}))$, symmetric gadget $(O(h))$, Yoshida-type gadget $(O(h))$.}
    \label{fig:Gadgets_attractors}
\end{figure}

\subsection{Numerical Integrators}
\label{ap:Numerical_Integrators}

\subsubsection*{RK4 Method}

Let
\[
\dot{y} = f(t,y), \qquad y_n \approx y(t_n), \qquad h = t_{n+1}-t_n.
\]
The classical four-stage Runge--Kutta method computes the internal stages
\begin{align}
    k_1 &= f(t_n,y_n), \\
    k_2 &= f\qty(t_n+\frac{h}{2},\, y_n+\frac{h}{2}k_1), \\
    k_3 &= f\qty(t_n+\frac{h}{2},\, y_n+\frac{h}{2}k_2), \\
    k_4 &= f\qty(t_n+h,\, y_n+h k_3),
\end{align}
followed by the update
\begin{equation}
    y_{n+1} = y_n + \frac{h}{6}\qty(k_1 + 2k_2 + 2k_3 + k_4).
\end{equation}
For smooth vector fields this gives a local truncation error of order $O(h^5)$ and therefore a global error of order $O(h^4)$.

When used to prolong a diffeomorphism as considered in the main text, RK4 is applied to a closed base-space system in variables $z=(x,u)$ or $z=(q,u)$. If the resulting numerical base map is written as $\varphi_h(z_n) = (\chi,\psi)$ then the momentum $p$ is reconstructed by the prolongation formula
\begin{equation}
    p_{n+1} = \frac{\partial_x \psi + p_n\,\partial_u \psi}{\partial_x \chi + p_n\,\partial_u \chi}.
\end{equation}
Equivalently, if $J = D\varphi_h(z_n)$ is the Jacobian of the base map, written as
\[
J = \begin{pmatrix} J_{00} & J_{01} \\ J_{10} & J_{11} \end{pmatrix},
\]
then
\begin{equation}
    p_{n+1} = \frac{J_{10} + p_n J_{11}}{J_{00} + p_n J_{01}}.
\end{equation}
This is the form used in the lifted RK4 experiments for the Van der Pol oscillator and for the prolonged DHO tests.

\medskip
\begingroup
\small
\setlength{\tabcolsep}{4pt}
\renewcommand{\arraystretch}{1.15}
\begin{center}
\captionsetup{hypcap=false}
\captionof{table}{Use of RK4 in the numerical experiments.}
\label{tab:appendix_rk4}
\begin{tabularx}{\textwidth}{|p{0.16\textwidth}|X|X|X|}
\hline
Experiment & Use & Parameter values & Initial conditions / horizon \\
\hline
Damped harmonic oscillator & Classical full-system RK4 baseline for the 3-D contact ODE. & $\gamma=0.3$; main sweep $h \in 10^{\mathrm{linspace}(-1,-2.5,15)}$; conformal-factor check uses $h=0.1$. & $(q_0,p_0,u_0)=(1,0,0)$; main comparison $T=20$; conformal-factor run $T=50$. \\
\hline
Damped harmonic oscillator & Lifted RK4 on the closed base problem $(q,u)$, followed by prolongation for $p$. & $\gamma=0.3$; $h \in 10^{\mathrm{linspace}(-1,-2.5,15)}$; local check at $h=0.1$. & $(q_0,p_0,u_0)=(1,0,0)$; $T=20$. \\
\hline
Forced Van der Pol oscillator & Lifted RK4 on the forced base system, followed by prolongation for $p$. & $\varepsilon=A=5$; regular case $\omega=2.466$, $h=0.02$, $T=200$; chaotic case $\omega=2.463$, $h=0.01$, $T=500$. & $(x_0,p_0,u_0)=(1,0,0)$. \\
\hline
\end{tabularx}
\end{center}
\endgroup

\subsubsection*{DOP Method}

Throughout the numerical experiments, ``DOP'' refers to the adaptive explicit Dormand--Prince method DOP853. As with any explicit embedded Runge--Kutta method, one computes internal stages of the form
\begin{equation}
    k_i = f\qty(t_n + c_i h_n,\, y_n + h_n \sum_{j<i} a_{ij} k_j),
\end{equation}
and then forms a high-order update
\begin{equation}
    y_{n+1}^{(8)} = y_n + h_n \sum_{i=1}^s b_i k_i.
\end{equation}
An embedded lower-order approximation
\begin{equation}
    \widehat{y}_{n+1} = y_n + h_n \sum_{i=1}^s \widehat{b}_i k_i
\end{equation}
provides the local defect estimate
\begin{equation}
    e_n = y_{n+1}^{(8)} - \widehat{y}_{n+1},
\end{equation}
which is then used to adapt the next step size according to a controller of the form
\begin{equation}
    h_{n+1} = s\,h_n\qty(\frac{\mathrm{tol}}{\norm{e_n}})^{1/8},
\end{equation}
where $s\in(0,1)$ is a safety factor and the exponent reflects the principal order of the accepted method.

In this work, DOP853 is used in two closely related ways. First, it serves as a high-accuracy reference solver for the full contact system or for the closed base-space system, depending on the example. Second, in the lifted adaptive Van der Pol experiments, DOP853 is applied to the augmented variational system
\begin{equation}
    \dot{z} = f(z,t), \qquad \dot{J} = Df(z,t)J, \qquad J(t_n)=I,
\end{equation}
so that the Jacobian of the numerical base map is obtained simultaneously with the state update. The contact momentum is then recovered from the same prolongation identity,
\begin{equation}
    p_{n+1} = \frac{J_{10} + p_n J_{11}}{J_{00} + p_n J_{01}}.
\end{equation}
Thus the distinction between lifted RK4 and lifted DOP853 lies entirely in the numerical base integrator and in the accuracy with which the Jacobian is computed; the contact reconstruction step is identical.

\medskip
\begingroup
\small
\setlength{\tabcolsep}{4pt}
\renewcommand{\arraystretch}{1.15}
\begin{center}
\captionsetup{hypcap=false}
\captionof{table}{Use of DOP853 in the numerical experiments.}
\label{tab:appendix_dop}
\begin{tabularx}{\textwidth}{|p{0.16\textwidth}|X|X|X|}
\hline
Experiment & Use & Solver settings / parameter values & Initial conditions / horizon \\
\hline
Damped harmonic oscillator & High-accuracy reference solver for the full 3-D contact ODE, used through the shared \texttt{reference\_solution} helper. & DOP853 with $\mathrm{rtol}=10^{-13}$, $\mathrm{atol}=10^{-15}$, dense output, and default $n_{\mathrm{eval}}=10001$; experiment parameter $\gamma=0.3$. & $(q_0,p_0,u_0)=(1,0,0)$; used at $T=20$ in the main studies and at $T=50$ in the conformal-factor comparison. \\
\hline
Forced Van der Pol oscillator & Base-space reference and lifted DOP853 forced-attractor integrations. & Base reference uses DOP853 with $\mathrm{rtol}=10^{-13}$, $\mathrm{atol}=10^{-15}$; lifted forced DOP853 uses $\mathrm{rtol}=10^{-10}$, $\mathrm{atol}=10^{-12}$ with outer step $h=0.1$; parameters $\varepsilon=A=5$, $\omega=2.466$ (regular) and $\omega=2.463$ (chaotic). & $(x_0,p_0,u_0)=(1,0,0)$; base-reference runs use $T=200$ and $T=500$; lifted-DOP853 attractor runs use $T=300$ and $T=600$. \\
\hline
Nonlinear dissipative double-well & High-accuracy reference solver for the full 3-D contact system in 02\_double\_well\_p2u\_splitting; DOP853 is attempted first, with Radau fallback when the Riccati component becomes too stiff. & DOP853 and Radau both use $\mathrm{rtol}=10^{-12}$, $\mathrm{atol}=10^{-14}$, and $\mathrm{max\_step}=0.05$; stability probes use $\sigma \in \{0.5,1,2,5\}$ with $T \in \{20,50,100,200\}$; phase-portrait references use $\sigma \in \{0,0.5,1,2\}$ with $T=150$. & $(x_0,p_0,u_0)=(0.5,1,0)$; additional attractor projections use $T=300$ for $\sigma=0.5$, $T=200$ for $\sigma=1$, and $T=150$ for $\sigma=2$. \\
\hline
\end{tabularx}
\end{center}
\endgroup

\subsubsection*{Bernoulli ODE Integrator}

Throughout these experiments, the ``Bernoulli ODE integrator'' is not a generic approximate time-stepping scheme like RK4 or DOP853, but rather an exact closed-form substep used when a split contact Hamiltonian reduces to a scalar Bernoulli-type evolution. In the present draft, this occurs in the nonlinear dissipative double-well example, where the dissipative term $\sigma p^2u$ and the effective kinetic term $\frac12(1+2\sigma u)p^2$ both admit analytic updates. The only inputs are the substep size $\tau$ and the model parameter $\sigma$; there are no tolerances or iterative solves.

For the contact sub-Hamiltonian
\begin{equation}
    H_B(x,p,u) = \sigma p^2u,
\end{equation}
the contact equations are
\begin{equation}
    \dot{x} = 2\sigma pu, \qquad \dot{p} = -\sigma p^3, \qquad \dot{u} = \sigma p^2u.
\end{equation}
Here $p$ satisfies a Bernoulli equation of cubic type, and the invariant $pu$ allows the remaining variables to be recovered exactly. Over one substep of length $\tau$, the map is
\begin{equation}
    p_{n+1} = \frac{p_n}{\sqrt{1 + 2\sigma p_n^2\tau}}, \qquad
    u_{n+1} = u_n\sqrt{1 + 2\sigma p_n^2\tau}, \qquad
    x_{n+1} = x_n + 2\sigma p_nu_n\tau.
\end{equation}
This is the exact Bernoulli contact step used in the CSC splitting for this system.

The TV splitting for the same example uses the effective kinetic Hamiltonian
\begin{equation}
    H_T(x,p,u) = \frac12(1+2\sigma u)p^2,
\end{equation}
for which
\begin{equation}
    \dot{x} = (1+2\sigma u)p, \qquad \dot{p} = -\sigma p^3, \qquad \dot{u} = \frac12(1+2\sigma u)p^2.
\end{equation}
In this case, $(1+2\sigma u)p$ is conserved along the subflow. Writing
\begin{equation}
    D_n = \sqrt{1 + 2\sigma p_n^2\tau},
\end{equation}
the exact update becomes
\begin{equation}
    p_{n+1} = \frac{p_n}{D_n}, \qquad
    u_{n+1} = \frac{(1+2\sigma u_n)D_n - 1}{2\sigma}, \qquad
    x_{n+1} = x_n + (1+2\sigma u_n)p_n\tau.
\end{equation}
This exact Bernoulli map is what makes the TV splitting fully explicit and second-order once composed symmetrically with the exact potential kick. The same analytic substeps are reused in the conformal-factor follow-up experiments for this system.

\subsubsection*{Higher-Order Commutator Gadgets}

For the nonlinear dissipative double-well example with
\begin{equation}
    H(x,p,u) = H_s(x,p,u) + \sigma p^2u, \qquad H_s(x,p,u) = \frac12 p^2 + (x^2-1)^2,
\end{equation}
the basic commutator construction writes the nonlinear contact term as
\begin{equation}
    p^2u = [A,B], \qquad A = -\frac12 u^2, \qquad B = p^2,
\end{equation}
and then approximates the corresponding bracket flow by a four-subflow gadget $G(\epsilon)$ built from the exact flows of $A$ and $B$. In the notebook implementation, $G(\epsilon)$ is calibrated so that
\begin{equation}
    G(\epsilon) = \exp\qty(\epsilon^2[A,B] + O(\epsilon^3)).
\end{equation}
Since $\epsilon \sim \sqrt{\sigma t_c}$ for a contact substep of duration $t_c$, the basic gadget used in Splitting C has local commutator error $O(\epsilon^3)=O(t_c^{3/2})$, which leads to global order $\frac12$ after composition with the St\"ormer--Verlet step for $H_s$.

The two higher-order variants used in the extended comparison for this same $p^2u$ example keep the same CSC-Strang outer structure, but replace the basic gadget by deeper real compositions that cancel the cubic Baker--Campbell--Hausdorff term.

The first improved variant is the \emph{symmetric gadget} (Splitting D). If $t_c$ denotes one contact half-step and $m=m_{\mathrm{gadget}}$ is the optional substepping parameter, the notebook defines
\begin{equation}
    \Gamma_D(t_c) = \qty(G(-s) \circ G(s))^m, \qquad s = \sqrt{\frac{\sigma t_c}{2m}}.
\end{equation}
By symmetry, the odd BCH terms cancel, so the gadget error improves from $O(\epsilon^3)$ to $O(\epsilon^4)$, that is, from $O(t_c^{3/2})$ to $O(t_c^2)$. The full method is then
\begin{equation}
    \Phi_D(\tau) = \Gamma_D\qty(\frac{\tau}{2}) \circ \Phi^{\mathrm{SV}}_{H_s}(\tau) \circ \Gamma_D\qty(\frac{\tau}{2}),
\end{equation}
where $\Phi^{\mathrm{SV}}_{H_s}$ is the standard St\"ormer--Verlet step for $H_s$. This yields a globally first-order contact integrator. Each call to $\Gamma_D$ uses $8m$ exact $A/B$ subflows.

The second improved variant is the \emph{Yoshida triple-jump gadget} (Splitting E), defined by
\begin{equation}
    \Gamma_E(t_c) = \qty(G(\gamma_1\epsilon) \circ G(\gamma_0\epsilon) \circ G(\gamma_1\epsilon))^m,
    \qquad \epsilon = \sqrt{\frac{\sigma t_c}{m}},
\end{equation}
with coefficients
\begin{equation}
    \gamma_1 = \frac{1}{\sqrt{2 + 2^{2/3}}}, \qquad \gamma_0 = -2^{1/3}\gamma_1.
\end{equation}
These satisfy
\begin{equation}
    2\gamma_1^2 + \gamma_0^2 = 1, \qquad 2\gamma_1^3 + \gamma_0^3 = 0,
\end{equation}
so the leading cubic BCH term is cancelled and the contact approximation again has per-step error $O(\epsilon^4)=O(t_c^2)$. The full integrator is
\begin{equation}
    \Phi_E(\tau) = \Gamma_E\qty(\frac{\tau}{2}) \circ \Phi^{\mathrm{SV}}_{H_s}(\tau) \circ \Gamma_E\qty(\frac{\tau}{2}),
\end{equation}
which is also globally first-order. Its cost is higher, however: each call to $\Gamma_E$ requires $12m$ exact $A/B$ subflows.

In the higher-order convergence and phase-portrait comparisons for this system, both Splittings D and E were run with $m_{\mathrm{gadget}}=4$. The notebook also includes a direct bracket-verification test, which fits $\epsilon$-rates $3.99$ and $4.00$ for the symmetric and Yoshida gadgets respectively, confirming the expected $O(\epsilon^4)$ local commutator error. Thus both methods improve the naive gadget from global order $\frac12$ to global order $1$, although the symmetric gadget achieves the same observed order at lower cost than the Yoshida composition.

\end{document}